\documentclass[a4paper,11pt]{article}

\usepackage{ucs}
\usepackage[utf8x]{inputenc}
\usepackage{type1ec}
\usepackage{hyperref}
\usepackage{fullpage}

\usepackage{bbm}
\usepackage[bbgreekl]{mathbbol}
\usepackage{nccrules}
\usepackage{tocbibind}
\usepackage[intlimits,sumlimits,leqno]{amsmath}
\usepackage{amsthm}
\usepackage{amssymb}
\usepackage{mathrsfs}
\usepackage{stmaryrd}
\usepackage{xspace}
\usepackage[cmtip,all]{xy}

\newcommand{\Z}{\ensuremath{\mathbb{Z}}}

\newcommand{\R}{\ensuremath{\mathbb{R}}}
\newcommand{\C}{\ensuremath{\mathbb{C}}}

\newcommand{\Aut}{\ensuremath{\mathrm{Aut}}}
\newcommand{\Tr}{\ensuremath{\mathrm{tr}\,}}

\newcommand{\Hy}{\ensuremath{\mathbb{H}}}  % Hyperbolic plane

\newcommand{\dd}{\ensuremath{\,\mathrm{d}}}

\newcommand{\angles}[1]{\ensuremath{\langle #1 \rangle}}

\newcommand{\sgn}{\ensuremath{\mathrm{sgn}}}
\newcommand{\Stab}{\ensuremath{\mathrm{Stab}}}
\newcommand{\extotimes}{\ensuremath{\boxtimes}} % Exterior tensor product

\newcommand{\identity}{\ensuremath{\mathrm{id}}}

\newcommand{\Hom}{\ensuremath{\mathrm{Hom}}}
\newcommand{\Isom}{\ensuremath{\mathrm{Isom}}}
\newcommand{\End}{\ensuremath{\mathrm{End}}}
\newcommand{\rightiso}{\ensuremath{\stackrel{\sim}{\rightarrow}}}

\newcommand{\Ker}{\ensuremath{\mathrm{Ker}\,}}

\newcommand{\Ad}{\ensuremath{\mathrm{Ad}\,}}

\newcommand{\Gm}{\ensuremath{\mathbb{G}_\mathrm{m}}}
\newcommand{\Ga}{\ensuremath{\mathbb{G}_\mathrm{a}}}

\newcommand{\Res}{\ensuremath{\mathrm{Res}}}

\newcommand{\GL}{\ensuremath{\mathrm{GL}}}
\newcommand{\tGL}{\ensuremath{\widetilde{\mathrm{GL}}}}  % Twisted GL
\newcommand{\SO}{\ensuremath{\mathrm{SO}}}
\newcommand{\Or}{\ensuremath{\mathrm{O}}}

\newcommand{\U}{\ensuremath{\mathrm{U}}}
\newcommand{\SU}{\ensuremath{\mathrm{SU}}}
\newcommand{\PGL}{\ensuremath{\mathrm{PGL}}}
\newcommand{\gl}{\ensuremath{\mathfrak{gl}}}
\newcommand{\so}{\ensuremath{\mathfrak{so}}}

\newcommand{\Sp}{\ensuremath{\mathrm{Sp}}}

\newcommand{\Ind}{\ensuremath{\mathrm{Ind}}}

\newcommand{\Out}{\ensuremath{\mathrm{Out}}}

\newcommand{\Lgrp}[1]{\ensuremath{{}^{\mathrm{L}} #1}}
\newcommand{\WD}{\ensuremath{\mathrm{WD}}} % Weil-Deligne group

\theoremstyle{plain}
\newtheorem{proposition}{Proposition}[subsection]
\newtheorem{lemma}[proposition]{Lemma}
\newtheorem{theorem}[proposition]{Theorem}
\newtheorem{corollary}[proposition]{Corollary}

\theoremstyle{definition}
\newtheorem{definition}[proposition]{Definition}
\newtheorem{definition-theorem}[proposition]{Definition-Theorem}
\newtheorem{definition-proposition}[proposition]{Definition-Proposition}
\newtheorem{hypothesis}[proposition]{Hypothesis}

\newtheorem{remark}[proposition]{Remark}

\newcommand{\GSnorm}{\ensuremath{\stackrel{\mathrm{GS}}{\longrightarrow}}} % Goldberg-Shahidi norm

\title{On a pairing of Goldberg-Shahidi for even orthogonal groups}
\author{Wen-Wei Li}
\date{}

\begin{document}

\maketitle

\begin{abstract}
  Let $\pi \extotimes \sigma$ be a supercuspidal representation of $\GL(2n) \times \SO(2n)$ over a $p$-adic field with $\pi$ selfdual, where $\SO(2n)$ stands for a quasisplit even special orthogonal group. In order to study its normalized parabolic induction to $\SO(6n)$, Goldberg and Shahidi defined a pairing $R$ between the matrix coefficients of $\pi$ and $\sigma$ which controls the residue of the standard intertwining operator. The elliptic part $R_\text{ell}$ of $R$ is conjectured to be related to twisted endoscopic transfer. Based on Arthur's endoscopic classification and Spallone's improvement of Goldberg-Shahidi program, we will verify some of their predictions for general $n$, under the assumption that $\pi$ does not come from $\SO(2n+1)$.
\end{abstract}

%MSC 2010: 22E50 (Primary) ,11F70

\tableofcontents

\section{Introduction}\label{sec:intro}
\paragraph{History}
The residue of intertwining operators plays a pivotal role in the study of non-discrete tempered spectrum of reductive groups over local fields. To be precise, let $F$ be a non-archimedean local field of characteristic zero, and consider a quasisplit classical group $G_1$ together with a maximal proper Levi subgroup of the form $M := \GL(H) \times G$, where $H$ is some $F$-vector space and $G$ is a classical group of the same type as $G_1$. Choose a parabolic subgroup $P=MU$. Let $\mathcal{I}_P(\pi \extotimes \sigma)$ be the normalized parabolic induction, where $\pi \extotimes \sigma$ is an essentially square-integrable irreducible representation of $M(F)$. To study its reducibility, one may assume $\pi$ to be selfdual. After Harish-Chandra, the problem is reduced to the study of the residue at $\lambda=0$ of the standard intertwining operator $J_P(w_0, (\pi \extotimes \sigma)_\lambda)$ for $\mathcal{I}_P((\pi \extotimes \sigma)_\lambda)$, where
$\pi \extotimes \sigma \mapsto (\pi \extotimes \sigma)_\lambda$ means the twist by the unramified character $M(F)$ attached to $\lambda \in \mathfrak{a}_{M,\C}^*$, and $w_0$ is a suitable element in $N_{G_1(F)}(M(F))$.

In the case of split even orthogonal groups and supercuspidal $\pi \extotimes \sigma$, Shahidi introduced a notion of norm correspondence in \cite{Sh95} that relates conjugacy classes between $\GL(H)$ and $G$ (baptized GS-norm in this article). Using a lemma of Rallis \cite[Lemma 4.1]{Sh92}, he showed that $\Res_{\lambda=0} J_P(w_0, (\pi \extotimes \sigma)_\lambda)$ is actually governed by a pairing $R$ between matrix coefficients of $\pi$ and $\sigma$. The main idea is to decompose $U$ into $M$-orbits. This method is polished and generalized to other classical groups in a series of papers by Goldberg and Shahidi \cite{GS98,GS01,GS09}. Later on, Spallone rewrote $R$ as a ``weighted integral'' of matrix coefficients, and gave an amended formula in \cite{Sp08,Sp11}. Recently, a broader interpretation in the context of generalized functionals and Bessel functions is given in \cite[\S 6.3]{CS11}.

In either formulation, $R$ can be written as the sum of a regular, or elliptic term $R_\text{ell}$ and a singular term $R_\text{sing}$. The term $R_\text{ell}$ can be expressed as an integral pairing between the character of $\sigma$ and some twisted character of $\pi$. Shahidi conjectured that the nonvanishing of $R_\text{ell}$ should be closely related to twisted endoscopy for $\GL(H)$, as developed in \cite{KS}.

The reducibility of $\mathcal{I}_P(\pi \extotimes \sigma)$ can be determined by Arthur's results \cite[\S 6.6]{ArEndo} nowadays. Nevertheless, some conjectures about $R_\text{ell}$ remain unanswered. Let us explain.

We will concentrate on the case $G_1 = \SO(V_1,q_1)$, $G=\SO(V,q)$ where $(V_1,q_1)$ and $(V,q)$ are even-dimensional $F$-quadratic spaces, such that $\dim_F V = \dim_F H = 2n$, i.e.\ $M$ consists of matrices of ``three equal-sized blocks''. We exclude the case $\SO(V,q) \simeq \GL(1)$. The pairing $R_\text{ell}$ essentially takes the form
$$ \sum_{\substack{T: \text{elliptic} \\ /\text{conjugation}}} |W(\SO(V,q), T(F))|^{-1} \int_{T(F)} I^{\SO(V,q)}(\sigma,\gamma) I^{\tGL(H)}(\tilde{\pi}, \delta) \dd\gamma, $$
where
\begin{itemize}
  \item $I^{\SO(V,q)}(\sigma,\gamma)$ is the normalized character of $\sigma$;
  \item $I^{\tGL(H)}(\tilde{\pi}, \delta)$ is the normalized twisted character of $\pi$;
  \item $\gamma \mapsto \delta$ is a section for GS-norm.
\end{itemize}
For further explanations, see \S\ref{sec:pairing}.

This looks like the elliptic inner product for square-integrable representations \cite[Theorem 3]{Cl91}, but the groups here are different, and a twist intervenes. Shahidi made the following definition (sic): a supercuspidal selfdual representation $\pi$ of $\GL(H)$ is called a twisted endoscopic transfer of a supercuspidal representation $\sigma$ of $\SO(V,q)$, if the pairing $R_\text{ell}$ for $\pi \extotimes \sigma$ is not identically zero. This is \cite[Definition 5.1]{GS98}.

Indeed, Goldberg and Shahidi have shown that the GS-norm is compatible with the norm mapping in twisted endoscopy, defined by Kottwitz and Shelstad in \cite[\S 3]{KS}. On the other hand, the Langlands-Shahidi method via $L$-functions also gives some evidence for this definition when $\sigma$ is generic; see \cite[\S 5]{GS98}. The connection remains conjectural, however, until the emergence of \cite{SS10}. Let $\omega_\pi$ denote the central character of $\pi$. In the case $n=1$ and $\SO(V,q) = E^1 := \{x \in E^\times : N_{E/F}(x)=1\}$ where $E/F$ is a quadratic extension, Shahidi and Spallone showed that
\begin{itemize}
  \item if $\omega_\pi = 1$, then $R_\text{ell}$ is not identically zero for all $\pi$ and $\sigma$;
  \item if $\omega_\pi \neq 1$, then $R_\text{ell}$ is not identically zero if and only if $\pi$ is attached to $\Ind_{E/F}(\sigma)$ (the local automorphic induction), where $\sigma$ is regarded as a representation of $E^\times$ via the usual surjection $E^\times \to E^1$.
\end{itemize}
The proof of the first assertion is a direct calculation using Shimizu's explicit character formulae, while that of the second assertion relies upon the character identities of Labesse-Langlands, which amounts to the twisted endoscopy for $\GL(2)$ (cf.\ \cite[\S 5.3]{KS}). One of their key observations is that the Labesse-Langlands transfer factor $\Delta(\gamma,\delta)$ is independent of $\delta$, provided that $\gamma$ and $\delta$ are related by GS-norm. Also note that $\omega_\pi = 1$ if and only if $\pi$ comes from $\PGL(2)=\SO(3)$.

\paragraph{Goal of this article}
We consider similar problems for general $\GL(H)$ and $\SO(V,q)$ with $\dim_F H = \dim_F V = 2n$, under the Hypothesis \ref{hyp:even-SO} that $\pi$ does not come from $\SO(2n+1)$ (the split form) by twisted endoscopic transfer. The main impetus comes from Arthur's monumental work \cite{ArEndo}, which provides the necessary local Langlands correspondence and twisted character relations. Our main result is Theorem \ref{prop:main} that reconciles Shahidi's definition and Arthur's endoscopic classification for $\SO(V,q)$.

The implications in number theory and harmonic analysis can be understood in terms of $L$-functions:
$$\xymatrix{
  *+[F]\txt{nonvanishing of \\ $R=R_\text{ell} + R_\text{sing}$} & *+[F]\txt{pole at $\lambda=0$ of \\ $J_P(w_0, (\pi \extotimes \sigma)_\lambda)$} \ar@{<->}[l] \ar@{<->}[r] & *+[F]\txt{pole at $s=0$ of \\ $L(s, \pi \times \sigma) L(2s, \pi, \wedge^2)$}
} \; ,$$
the rightmost arrow comes from Langlands' conjecture on the normalization of intertwining operators, which is also a good vehicle for studying such $L$-functions as illustrated in \cite[Chapter 8]{Sh10} for generic inducing data. In view of our result, the nonvanishing of $R_\text{ell}$ should point to a pole of the Rankin-Selberg $L$-function $L(s, \pi \times \sigma)$ at $s=0$, provided that the exterior square $L$-function $L(s, \pi, \wedge^2)$ is holomorphic at $s=0$. Cf. \cite{GS98}.

The idea of the proof is to apply the endoscopic character relations, then use Schur's orthogonality relations on $\SO(V,q)$. This does not follow from \cite{ArEndo} for free since the twisted transfer factor $\Delta(\gamma,\delta)$ intervenes. Let us explain the bottlenecks.
\begin{enumerate}
  \item Instead of the language of \cite{KS}, we adopt Labesse's notion of twisted spaces \cite{Lab04} systematically. That is, we will do the harmonic analysis on the $\GL(H)$-bitorsor of non-degenerate bilinear forms on $H$, or equivalently the space $\Isom(H,H^\vee)$, denoted by $\tGL(H)$.
  \item Following \cite{Sp11}, the Goldberg-Shahidi-Spallone formalism is reformulated in a basis-free way. The twisted space $\tGL(H)$ comes out naturally from this perspective. One also obtains a transparent description of GS-norms for very regular classes.
  \item Unlike some other applications of endoscopy, knowing the formal properties of transfer factors does not suffice. The key ingredient comes from Waldspurger's elegant formula \cite{Wa10} of the transfer factor $\Delta$ for $\SO(V,q)$, viewed as an elliptic endoscopic group of $\tGL(H)$. Once the twisted paraphrase of Goldberg-Shahidi-Spallone formalism is in place, it will follow easily that  $\Delta$ only depends on $(V,q)$.
\end{enumerate}
The applicability of Waldspurger's formula is inextricably liked with the first two points. As a consequence, our notations will deviate somehow from those of \cite{GS98,GS01,GS09,SS10,Sp11}. We hope the reader will be convinced of the flexibility of twisted spaces, especially in the context of classical groups of higher rank.

In this article, we restrict ourselves to the elliptic terms in Goldberg-Shahidi or Spallone's formula for $R$. This is certainly unsatisfactory. If the non-elliptic terms in $R$ can be expressed in terms of Arthur's weighted orbital integrals, as alluded in \cite{SS10}, then the endoscopic character relations obtained in \cite{ArEndo} might still be applicable by invoking \cite{Ar87}.

\paragraph{Structure of this article}
In \S\ref{sec:notations}, we set up the basic terminologies, including the conventions of quadratic and hermitian forms that will be heavily used in \S\ref{sec:GSS-formalism}.

Due to the temporary lack of a comprehensive exposition on Labesse's notion of twisted spaces, we collect some basic notions and results in \S\ref{sec:twisted}. Although we will only encounter the twisted space of bilinear or sesquilinear forms, a general introduction seems profitable. The main sources are \cite{Wa09-ep,L10}.

In \S\ref{sec:twisted-endoscopy}, we review the geometric aspect of twisted endoscopy for $\tGL(2n)$ in the spirit of \cite{ArEndo}. The emphasis is put on the simple endoscopic groups $\SO(V,q)$ where $\dim V = 2n$. Nevertheless, it seems more reasonable, and not too laborious, to include the other elliptic endoscopic data as well. In order to use Waldpsurger's formulae \cite{Wa10}, we also need to parametrize conjugacy classes in terms of linear algebra and calculate some invariants attached to regular nilpotent orbits. After these apéritifs, we are able to state the formula for geometric transfer factors in \S\ref{sec:Delta-formula1}. We will also prove the Lemma \ref{prop:separation} that supersedes \cite[Proposition 8]{SS10}.

In \S\ref{sec:crude-LLC} we review the spectral aspect of twisted endoscopy. After a recollection of Arthur's central results, we give a formula of character values at elliptic classes. The idea of using Weyl integration formula to relate character values is certainly well-known, however the twisted case has not been worked out in detail. We will give a proof in the spirit of \cite{Ar96}. Note that we only need the crude version of local Langlands correspondence for square-integrable representations of $\SO(V,q)$, that describes the representations only up to $\Or(V,q)$. Indeed, the Goldberg-Shahidi pairing is equally ``crude'' in this respect.

The formalism of Goldberg-Shahidi-Spallone is reformulated in \S\ref{sec:GSS-formalism} for classical groups in general. Using the parametrization of very regular classes, we are able to relate GS-norms and the correspondence in \cite{Wa10} directly. As an easy corollary, we prove the constancy of transfer factors for elements $\gamma$, $\delta$ that are related via GS-norm, in the even orthogonal case. We give a rapid review of the elliptic term of Goldberg-Shahidi pairing $R$ in \S\ref{sec:pairing}. What we consider is actually the pairing $R^\text{ell}$ defined in \eqref{eqn:Rell} that is proportional to the $R_\text{ell}$ in Spallone's formula. Using the aforementioned results, Theorem \ref{prop:main} follows immediately. See Remark \ref{rem:nonqs} for the non-quasisplit cases.

Some statements in this article are evidently more general then needed. This is done intentionally, in the hope that they might be useful for further generalizations.

\paragraph{Acknowledgements}
The author is grateful to Li Cai, Kuok Fai Chao and Bin Xu for suggesting this topic. He would also like to thank Jean-Loup Waldspurger for discussions on transfer factors, as well as Wee Teck Gan, Freydoon Shahidi and Steven Spallone for their useful remarks. He acknowledges the referee for valuable comments.

\section{Notations and conventions}\label{sec:notations}
\paragraph{Local fields}
Unless otherwise specified, $F$ denotes a non-archimedean local field of characteristic zero with a fixed uniformizer $\varpi_F$. Let $\mathfrak{o}_F$, $\mathfrak{p}_F$ be the ring of integers of $F$ and its maximal ideal, respectively, and let $q_F$ be the cardinality of $\mathfrak{o}_F/\mathfrak{p}_F$. Denote by $|\cdot|$ the normalized absolute value on $F$. Fix an algebraic closure $\bar{F}$ of $F$. We will denote by $\Gamma_F$ its absolute Galois group, $\mathrm{W}_F$ its Weil group, and $\WD_F := \mathrm{W}_F \times \SU(2)$ its Weil-Deligne group.

Throughout this article, we fix an additive unitary character $\psi_F: F \to \C^\times$.

Let $L/F$ be a finite extension. We will write $N_{L/F}$ and $\Tr_{L/F}$ for the norm and trace, respectively. More generally, if $L$ is an étale $F$-algebra of finite dimension, we can still define $N_{L/F}$ and $\Tr_{L/F}$.

\paragraph{Group theory}
By $F$-group we mean a group variety over $F$. For any $F$-group $H$, denote by $H^\circ$ its identity connected component, and $H(F)$ denotes the set of $F$-points of $H$, equipped with the topology induced by $F$. By a subgroup of $H$ we mean a closed $F$-subgroup, unless otherwise specified. The normalizers in $H$ (resp. centralizers) will be denoted by $N_H(\cdot)$ (resp. $Z_H(\cdot)$) and the center of $H$ will be denoted by $Z_H$. The adjoint action by an element $x$ is denoted by $\Ad_x(\cdot)$. The same notations pertain to topological groups. Fraktur letters are used to denote Lie algebras.

Let $(\pi, V)$ be a smooth representation of a locally profinite group, its contragredient representation is denoted by $(\pi^\vee, V^\vee)$.

\paragraph{Hermitian spaces}
Consider a pair $(E,\tau)$, where $E$ is a field of characteristic $\neq 2$ and $\tau$ is an involution of $E$. For a finite-dimensional $E$-vector space $V$, define its hermitian dual as the $E$-vector space
$$ V^\vee := {}^\tau \Hom_E(V,E), $$
where the superscript ${}^\tau$ means that its scalar multiplication $\star$ is twisted: $\alpha \star \lambda := \tau(\alpha)\lambda$ for all $\alpha \in E$, $\lambda \in V^\vee$. We have the canonical isomorphism $V \rightiso V^{\vee\vee}$, namely $v \mapsto \tau(\angles{\cdot, v})$. We shall always identify $V$ and $V^{\vee\vee}$ without further remarks.

For $X \in \Hom_E(V_1,V_2)$, its hermitian transpose is denoted by $\check{X} \in \Hom_E(V_2^\vee, V_1^\vee)$, characterized by $\angles{\check{v}_2, Xv_1} = \angles{\check{X}\check{v}_2, v_1}$.

By a $(E,\tau)$-sesquilinear form on $V$, we mean a bi-additive map $q: V \times V \to E$ such that $q(\alpha v|\beta v') = \tau(\alpha)q(v|v')\beta$ for all $\alpha,\beta \in E$, $v,v' \in V$. The $(E,\tau)$-sesquilinear forms on $V$ are in bijection with elements in $\Isom_E(V,V^\vee)$: to $\sigma: V \rightiso V^\vee$ we attach the sesquilinear form $q(v|v') = \angles{\sigma(v),v'}$.

Let $q$ be a sesquilinear form on $V$, its transposed form ${}^t q$ is defined by $(v,v') \mapsto \tau(q(v'|v))$.

Let $\epsilon=\pm 1$, a $(E,\tau)$-hermitian form of sign $\epsilon$ on $V$ is a non-degenerate $(E,\tau)$-sesquilinear form $q$ such that ${}^t q =\epsilon q$. We also call $(V,q)$ a $(E,\tau)$-hermitian space of sign $\epsilon$. A subspace $H \subset V$ is called totally isotropic if $q|_{H \times H}=0$.  Define $\U(V,q) := \Stab_{\GL_E(V)}(q)$.

Only two cases are encountered in this article: $E=F$, $\tau=\identity$ or $E/F$ is a quadratic extension, $\text{Gal}(E/F)=\{1,\tau\}$. In the first case, we get $F$-quadratic (resp. symplectic) spaces when $\epsilon=1$ (resp. $\epsilon=-1$); in the latter case, we will omit $\tau$ and speak of $E/F$-hermitian spaces, $E/F$-sesquilinear forms, etc. The main concern, however, will be the case of quadratic spaces.

\paragraph{Quadratic spaces}
As above, an $F$-quadratic space is a pair $(V,q)$. We set $q(v) := q(v|v)$. These two points of view on $q$, namely as a bilinear function on $V \times V$ or a quadratic function on $V$, will be used interchangeably.

An element in $a \in F^\times$ is said to be represented by $q$ if $a = q(v|v)$ for some $v \in V$. When there is no worry of confusions, we will drop $V$ and simply talk about the quadratic form $q$. The determinant of $q$ is $\det q := \det(q(e_i|e_j)_{i,j}) \in F^\times/F^{\times 2}$ where $\{e_1,\ldots, e_n\}$ is a basis of $V$; the discriminant of $q$ is $d_\pm(q) := (-1)^{\frac{n(n-1)}{2}} \det q$. Note that $d_\pm$ factors through the Witt group of $F$.

Up to isomorphism, there is a unique isotropic quadratic space (ie.\ $\exists v \neq 0$, $q(v)=0$) of dimension $2$; denote it by $\Hy$.

The orthogonal group (resp. special orthogonal group) is denoted by $\Or(V,q)$ (resp. $\SO(V,q)$).

For $a_1, \ldots, a_n \in F^\times$, write $\angles{a_1, \ldots, a_n}$ for the quadratic form on $F^n$ given by $(x_1, \ldots, x_n) \mapsto a_1 x_1^2 + \cdots a_n x_n^2$. For $c \in F^\times$, write $cq$ for the scaled quadratic form $(v,v') \mapsto cq(v|v')$. These notions can be generalized to the case where $F$ is replaced by an étale $F$-algebra.

We will make use of the Weil index $\gamma_{\psi_F}$ defined in \cite[\S 14]{Weil64}. Let $(V,q)$ be an $F$-quadratic space. The constant $\gamma_{\psi_F}$ is characterized by the identity
$$ \iint_{V \times V} \phi(x-y) \psi_F(q(v)) \dd x \dd y = \gamma_{\psi_F}(q) \int_V \phi(x) \dd x $$
for all Schwartz-Bruhat function $\phi$ on $V$, where we use the selfdual measure on $V$ with respect to the bi-character $\psi_F \circ 2q(\cdot|\cdot): V \times V \to \C^\times$. The map $\gamma_{\psi_F}$ induces a character of the Witt group of $F$; it only depends on the character of second degree $\psi_F \circ q: V \to \C^\times$.

The Hasse invariant $s(q) \in \{\pm 1\}$ of an $F$-quadratic form $(V,q)$ is defined as follows. Choose a diagonalization $(V,q) \simeq \angles{a_1, \ldots, a_m}$ and set
$$ s(q) := \prod_{1 \leq i < j \leq m} (a_i,a_j)_F $$
where $(\cdot,\cdot)_F$ is the quadratic Hilbert symbol of $F$.

\section{Twisted harmonic analysis}\label{sec:twisted}
\subsection{General notions}
We will systematically use the notion of twisted spaces introduced by Labesse \cite{Lab04}. In this section, we collect some basic notions that will be needed later. Our main references are \cite{L10} and \cite[\S 1]{Wa09-ep}.

\paragraph{Twisted spaces}
A twisted space is a pair $(G,\tilde{G})$, where $G$ is an $F$-group and $\tilde{G}$ is a $G$-bitorsor, that is,
\begin{itemize}
  \item $G$ acts on $\tilde{G}$ on the left and right, written multiplicatively as $(x,\delta) \mapsto x\delta$ and $(\delta,x) \mapsto \delta x$, respectively, where $x \in G$ and $\delta \in \tilde{G}$;
  \item the two actions commute: $(x\delta)y = x(\delta y)$ for all $x,y \in G$, $\delta \in \tilde{G}$;
  \item $\tilde{G}$ is a $G$-torsor under the action on either side.
\end{itemize}
One may also talk about the twisted subspaces $(H,\tilde{H}) \subset (G,\tilde{G})$, where $H$ is a subgroup of $G$ and $\tilde{H} \subset \tilde{G}$ forms a $H$-bitorsor.

There is an obvious notion of isomorphisms between twisted spaces. A twisted space $(G, \tilde{G})$ is called untwisted if it is isomorphic to $\tilde{G} = G$ on which $G$ acts by left and right multiplications. The group $G$ will often be omitted from the notations. As the definitions are purely categorical, one can also introduce twisted spaces in the category of locally compact spaces and define the functor $(G,\tilde{G}) \mapsto (G(F),\tilde{G}(F))$.

Let $\delta \in \tilde{G}$. Since $\tilde{G}$ is a bitorsor, there is an unique automorphism $\Ad_\delta: G \to G$ such that
$$ \delta x = \Ad_\delta(x) \delta, \quad x \in G .$$

One verifies that $\Ad_{x\delta y}=\Ad_x \Ad_\delta \Ad_y$ for all $x,y \in G$. The image of $\Ad_\delta$ in the outer automorphism group of $G$ is thus independent of $\delta$; denote it by $\theta := \theta_{\tilde{G}}$. The action of $\theta$ on $Z_G$ is well-defined. Thus we can set
$$ Z_{\tilde{G}} := Z_G^\theta = \{x \in G : x\delta = \delta x \text{ for all } \delta \in \tilde{G} \}. $$

Using the automorphisms $\Ad_\bullet$, define the centralizer and normalizer in $\tilde{G}$ of a subset $A \subset G$ as
\begin{align*}
  Z_{\tilde{G}}(A) & := \{\delta \in \tilde{G} : \Ad_\delta(a)=a, \text{ for all } a \in A \}, \\
  N_{\tilde{G}}(A) & := \{\delta \in \tilde{G} : \Ad_\delta(A)=A \}.
\end{align*}

As in the untwisted case, there is the adjoint action of $G$ on $\tilde{G}$ via $\delta \mapsto x\delta x^{-1}$. Its orbits are called the conjugacy classes in $\tilde{G}$. Set
\begin{align*}
  G^\delta & := \Stab_G(\delta), \\
  G_\delta & := \Stab_G(\delta)^\circ .
\end{align*}

\paragraph{Measures}
Assume $\tilde{G}(F) \neq \emptyset$. The bitorsor structure permits us to equip $\tilde{G}(F)$ with invariant measures. More precisely, let us fix $\delta_0 \in \tilde{G}(F)$. Given a left (resp. right) Haar measure $\mu$ on $G(F)$, we can transport it to $\tilde{G}(F)$ via the homeomorphism $G(F) \rightiso \tilde{G}(F)$, $x \mapsto x\delta_0$, thus obtain a left (resp. right) $G(F)$-invariant Radon measure $\mu \cdot \delta_0$ on $\tilde{G}(F)$. Similarly, by using the homeomorphism $x \mapsto \delta_0 x$, we obtain left (resp. right) measures $\delta_0 \cdot \mu$. In particular, $\tilde{G}(F)$ admits a bi-invariant measure if $G(F)$ is unimodular.

Let $\mu$ be a left or right Haar measure on $G(F)$. We have $\delta_0 \cdot \mu = \Delta_G(\Ad_{\delta_0}) \mu \cdot \delta_0$, where $\Delta_G(\Ad_{\delta_0}) \in \R_{>0}$ is the modulus of the automorphism $\Ad_{\delta_0}$, defined by
$$ \mu(f \circ \Ad_{\delta_0}) = \Delta_G(\Ad_{\delta_0}) \mu(f), \quad f \in C_c^\infty(G(F)). $$

When $G(F)$ is unimodular and $\Delta_G(\Ad_{\delta_0})=1$ for all $\delta_0$, one can readily verify that $\mu \cdot \delta_0 = \delta_0 \cdot \mu$ only depends on $\mu$. This will be the case when (1) $G$ is reductive and $\theta$ is of finite order, or (2) when $G$ is a torus.

Moreover, since $\tilde{G}(F)$ is homeomorphic to $G(F)$, the usual notion of $C_c^\infty$-functions and distributions carries over to $\tilde{G}(F)$. % See \cite{BZ76}

\subsection{Some structure theory}
Assume henceforth that $G$ is a connected reductive $F$-group and $\tilde{G}(F) \neq \emptyset$, and that the outer automorphism class $\theta$ is of finite order.

\paragraph{Parabolic and Levi subspaces}
A parabolic subspace of $(G,\tilde{G})$ is a pair $(P, \tilde{P})$ such that $P$ is a parabolic subgroup of $G$, $\tilde{P} := N_{\tilde{G}}(P)$ and we require $\tilde{P}(F) \neq \emptyset$. Since $N_G(P)=P$, we see that $(P,\tilde{P})$ forms a twisted space and $P$ is uniquely determined by $\tilde{P}$. A twisted Levi component of $\tilde{P}$ is a pair $(M,\tilde{M})$ where $M$ is a Levi component of $P$ and $(M,\tilde{M})$ is a twisted subspace of $(P,\tilde{P})$; the second condition is equivalent to $\tilde{M} := \tilde{P} \cap N_{\tilde{G}}(M)$. In this case, it can be shown that $\tilde{M}(F) \neq \emptyset$; see \cite[1.6]{Wa09}. Write $P=MU$ for the Levi decomposition where $U$ is the unipotent radical of $P$, then $\tilde{P}=\tilde{M}U$ accordingly.

A Levi subspace $(M,\tilde{M})$ of $(G,\tilde{G})$ is a Levi component of some parabolic subspace. Let $A_{\tilde{M}}$ denote the maximal split torus in $Z_{\tilde{M}}$. By \cite[\S 1.2]{Wa09-ep}, we have $M = Z_G(A_{\tilde{M}})$, $\tilde{M} = Z_{\tilde{G}}(A_{\tilde{M}})$, and all Levi subspaces of $\tilde{G}$ arise from split tori in this way.

\paragraph{Maximal tori}
A maximal torus of $(G,\tilde{G})$ is a pair $(T,\tilde{T})$ such that
\begin{itemize}
  \item $T$ is a maximal torus in $G$, $\tilde{T} \subset \tilde{G}$;
  \item there exists a Borel pair $(B,T)$ of $G \times_F \bar{F}$ such that $\tilde{T} = N_{\tilde{G}}(T) \cap N_{\tilde{G}}(B)$ over $\bar{F}$;
  \item $\tilde{T}(F) \neq \emptyset$.
\end{itemize}
Therefore $(T,\tilde{T})$ is a twisted subspace. Recall that we have a well-defined automorphism $\theta = \theta_{\tilde{T}}$ on $T$ since $T$ is commutative. Define $T^\theta$ as the subtorus of $\theta$-fixed points and $T_\theta := (T^\theta)^\circ$. We can also define the split torus $A_{\tilde{T}}$ as above. Then $(T,\tilde{T})$ is called $F$-elliptic if $A_{\tilde{T}} = A_{\tilde{G}}$. In the language of \cite[\S 2.12]{L10}, $\tilde{T}$ is called a maximal torus and $T_\theta$ is called a Cartan subgroup.

\paragraph{Regular semisimple elements}
The quasi-semisimple elements of $\tilde{G}$ are defined as those $\delta \in \tilde{G}$ such that $\Ad_\delta$ fixes a Borel pair of $G \times_F \bar{F}$. For such $\delta$, it is known that $G_\delta$ is reductive \cite[9.4]{St68}. For a quasi-semisimple $\delta \in \tilde{G}$, define its Weyl discriminant as
\begin{gather}\label{eqn:Weyl-discriminant}
  D^{\tilde{G}}(\delta) := \det(1-\Ad_\delta | \mathfrak{g}/\mathfrak{g}_\delta) .
\end{gather}

\begin{definition}
  An element $\delta \in \tilde{G}$ is called regular semisimple if $\delta$ is quasi-semisimple and $G_\delta$ is torus (cf. \cite[(2.10.4) and \S 2.11]{L10}). A regular semisimple element $\delta \in \tilde{G}$ is called strongly regular if $G^\delta$ is abelian.
\end{definition}
A fundamental fact is that the regular semisimple elements form a Zariski open dense subset of $\tilde{G}$ (see \cite[\S 2.11]{L10}), which is certainly invariant by the adjoint action. One should bear in mind that $G^\delta$ could be non-connected even for $\delta$ in general position.

In this article, we will only use strongly regular elements. Set
$$ \tilde{G}_\text{reg} := \{ \delta \in \tilde{G} : \text{semisimple strongly regular} \}. $$

The following construction can be found in \cite[\S 2.12 and \S 2.19]{L10}. For any regular semisimple $\delta \in \tilde{G}(F)$, let $T := Z_G(G_\delta)$, $\tilde{T} := T \delta$. Then $(T,\tilde{T})$ is a maximal torus. Conversely, every maximal torus arises in this way. Note that $G^\delta = T^\theta$, $G_\delta = T_\theta$. A regular semisimple element $\delta \in \tilde{G}(F)$ is called elliptic if its associated maximal torus $\tilde{T}$ is.

\subsection{Representations}\label{sec:twisted-rep}
\paragraph{Twisted representations}
Let $(G,\tilde{G})$ be a twisted space. A representation of $\tilde{G}(F)$ is a triplet $(\pi,\tilde{\pi},V)$, often abbreviated as $\tilde{\pi}$, where
\begin{itemize}
  \item $\pi: G(F) \to \Aut_\C(V)$ is a representation of $G(F)$;
  \item $\tilde{\pi}: \tilde{G}(F) \to \Aut_\C(V)$ is a map such that
  $$ \tilde{\pi}(x\delta y) = \pi(x)\tilde{\pi}(\delta)\pi(y) $$
  for all $x,y \in G(F)$ and $\delta \in \tilde{G}(F)$.
\end{itemize}
Two representations $(\pi_1,\tilde{\pi}_1,V_1)$, $(\pi_2,\tilde{\pi}_2,V_2)$ are called equivalent if there exist $\phi: V_1 \rightiso V_2$ and $A \in \Aut_\C(V_1)$ such that $\pi_1 = \varphi^{-1} \pi_2 \varphi$ and $\tilde{\pi}_1 = A \varphi^{-1} \tilde{\pi}_2 \varphi$. We call $(\pi, \tilde{\pi}, V)$ irreducible if there is no subspace $V' \subsetneq V$, $V' \neq \{0\}$ which is invariant by all $\tilde{\pi}(\delta)$. This is in general weaker than the irreducibility of $\pi$. If $\pi$ is irreducible, we will call $\tilde{\pi}$ strongly irreducible. We say that $\tilde{\pi}$ is smooth (resp. admissible) if $\pi$ is. In what follows, all representations of $\tilde{G}(F)$ are assumed to be smooth.

To a representation $(\pi, \tilde{\pi}, V)$ we can associate its contragredient $(\pi^\vee, \tilde{\pi}^\vee, V^\vee)$ as follows. Certainly, $(\pi^\vee, V^\vee)$ is the usual contragredient of $(\pi,V)$. For $\delta \in \tilde{G}(F)$, define $\tilde{\pi}^\vee(\delta)$ by $\angles{\tilde{\pi}^\vee(\delta)\check{v},v} = \angles{\check{v}, \tilde{\pi}(\delta)^{-1} v}$ for all $\check{v} \in V^\vee$, $v \in V$.

\paragraph{Characters}
Fix a Haar measure on $G(F)$. By choosing $\delta_0 \in \tilde{G}(F)$, we obtain a bi-invariant measure on $\tilde{G}(F)$ using the bitorsor structure. Let $\tilde{\pi}$ be an admissible representation of $\tilde{G}(F)$. For each $f \in C_c^\infty(\tilde{G}(F))$, set
$$ \Theta_{\tilde{\pi}}^{\tilde{G}}(f) := \Tr \left(\;\int_{\tilde{G}(F)} f(\delta)\tilde{\pi}(\delta) \dd \delta \right). $$
This is well-defined by the admissibility of $\tilde{\pi}$. This defines an invariant distribution on $\tilde{G}(F)$. The following result is due to Clozel.
\begin{theorem}[{\cite[Theorem 1]{Cl87}}]
  If $\tilde{\pi}$ is an irreducible admissible representation of $\tilde{G}(F)$, then the distribution $\Theta_{\tilde{\pi}}^{\tilde{G}}$ is a locally integrable function on $\tilde{G}(F)$, which is locally constant on $\tilde{G}_{\text{reg}}(F)$. Moreover, $|D^{\tilde{G}}|^{\frac{1}{2}} \Theta_{\tilde{\pi}}^{\tilde{G}}$ is locally bounded on $\tilde{G}(F)$.
\end{theorem}

Let $\tilde{\pi}$ be as above, we define the normalized character of $\tilde{\pi}$ as the locally bounded function on $\tilde{G}(F)$
$$ I^{\tilde{G}}(\tilde{\pi}, \cdot) :=  |D^{\tilde{G}}(\cdot)|^{\frac{1}{2}} \Theta_{\tilde{\pi}}^{\tilde{G}}(\cdot). $$

Note that as locally integrable functions, $\Theta_{\tilde{\pi}}^{\tilde{G}}$ and $I^{\tilde{G}}(\tilde{\pi},\cdot)$ are independent of the choice of measure on $\tilde{G}(F)$

\paragraph{Matrix coefficients}
Let $(\tilde{\pi},\pi,V)$ be a representation of $\tilde{G}(F)$. For $\check{v} \otimes v \in V^\vee \otimes V$, define the corresponding matrix coefficient as the smooth function on $\tilde{G}(F)$
$$ f_{\check{v} \otimes v}: \delta \mapsto \angles{\check{v}, \tilde{\pi}(\delta)v}. $$
This defines a linear map $V^\vee \otimes V \to C(\tilde{G}(F))$. Let $\mathcal{A}(\tilde{\pi})$ be its image.

Assume henceforth $\tilde{\pi}$ admissible and strongly irreducible, i.e.\ $\pi$ is irreducible. We know that $V^\vee \otimes V \rightiso \mathcal{A}(\tilde{\pi})$. By forgetting the twist, we also have the space of usual matrix coefficients $V^\vee \otimes V \rightiso \mathcal{A}(\pi)$. Hence there is an isomorphism $f \mapsto f^\circ$ from $\mathcal{A}(\tilde{\pi})$ to $\mathcal{A}(\pi)$ which satisfies $(f_{\check{v} \otimes v})^\circ(1) = \angles{\check{v}, v}$.

For $f \in \mathcal{A}(\tilde{\pi})$, one can check that $f(z^{-1} \delta z) = f(\delta)$ for all $z \in Z_G(F)$, $\delta \in \tilde{G}(F)$. To state the next result, we also have to fix a Haar measure on $Z_G(F)$. This permits to define the formal degree $d(\pi)$ when $\pi$ is essentially square-integrable modulo $Z_G(F)$.

\begin{proposition}\label{prop:orbint-character}
  Assume $\pi$ to be supercuspidal and irreducible. Let $f \in \mathcal{A}(\tilde{\pi})$. Let $\delta \in \tilde{G}_\text{reg}(F)$. If $\delta$ is elliptic in $\tilde{G}$, then
  $$ f^\circ(1) \Theta^{\tilde{G}}_{\tilde{\pi}}(\delta) = d(\pi) \int_{Z_G(F) \backslash G(F)} f(x^{-1} \delta x) \dd x. $$
\end{proposition}
\begin{proof}
  This is essentially a consequence of {\cite[(4.1.3)]{L10}}, in which the expression $f(x^{-1} \delta x)$ is replaced by
  $$ \text{vol}(K)^{-1} \int_K f(x^{-1}k^{-1} \delta kx) \dd k $$
  where $K$ is any open compact subgroup of $G(F)$. Our assumption on $\delta$ implies the compactness of $Z_G(F) \backslash Z_G(F) G_\delta(F) = Z_{\tilde{G}}(F) \backslash G_\delta(F)$. Since $x \mapsto f(x^{-1} \delta x)$ is integrable over $Z_G(F) G_\delta(F) \backslash G(F)$, it is also integrable over $Z_G(F) \backslash G(F)$. Hence the integral over $K$ can be absorbed into the outer one, as required.
\end{proof}
\begin{remark}
  It would be desirable to extend the result to $\pi$ square-integrable modulo $Z_G(F)$, cf.\ \cite[Proposition 5]{Cl91}. This requires more knowledge about the Schwartz-Harish-Chandra space in twisted setting.
\end{remark}

\subsection{Orbital integrals}\label{sec:orbint}
%FIXME: remove the definition of parabolic/Levi subspaces if need be.
\paragraph{Ordinary orbital integrals}
Let $\delta \in \tilde{G}_\text{reg}(F)$. It is known that the orbit $\{x \delta x^{-1} : x \in G(F)\}$ is closed in $\tilde{G}(F)$ and $G^\delta(F), G_\delta(F)$ are unimodular (see \cite[\S 5.1]{L10}). Choose a Haar measure on $G_\delta(F)$. The unnormalized orbital integral can thus be defined as the distribution
\begin{gather}\label{eqn:orbint}
  O^{\tilde{G}}_{\delta}(f) := [G^\delta(F) : G_\delta(F)]^{-1} \int_{G_\delta(F) \backslash G(F)} f(x^{-1} \delta x) \dd x, \qquad f \in C_c^\infty(\tilde{G}(F))
\end{gather}
where $G_\delta(F) \backslash G(F)$ is equipped with the quotient measure. The normalized orbital integral is defined by
$$ I^{\tilde{G}}(\delta,f) := |D^{\tilde{G}}(\delta)|^{\frac{1}{2}} O^{\tilde{G}}_\delta(f). $$

The following result is well-known. For the lack of an appropriate reference, we give a quick proof below.
\begin{proposition}\label{prop:orbint-bound}
  Let $f \in C_c^\infty(\tilde{G}(F))$. Fix a maximal torus $\tilde{T}$ of $\tilde{G}$ and fix a Haar measure on $T_\theta(F)$, which permits to define $I^{\tilde{G}}(\delta,f)$ for $\delta \in \tilde{G}_{\mathrm{reg}} \cap \tilde{T}(F)$. Then $\delta \mapsto I^{\tilde{G}}(\delta,f)$ is a locally bounded function on $\tilde{G}_{\mathrm{reg}} \cap \tilde{T}(F)$.
\end{proposition}
\begin{proof}
  By twisted descent for $O^{\tilde{G}}_{\delta}(f)$ \cite[2.4]{Wa08} and $|D^{\tilde{G}}(\delta)|^{\frac{1}{2}}$ (use \cite[1.1 and 1.3]{KS}), it suffices to prove a similar untwisted assertion on Lie algebras. Now apply \cite[Theorem 13]{HC70}.
\end{proof}

Now we can state the twisted Weyl integration formula. Let $(T,\tilde{T})$ be a maximal torus and set
\begin{align*}
  T(F)_{/\theta} & := (1-\theta)T(F) \backslash T(F), \\
  \tilde{T}(F)_{/\theta} & := (1-\theta)T(F) \backslash \tilde{T}(F).
\end{align*}

Then $(T(F)_{/\theta}, \tilde{T}(F)_{/\theta})$ is a twisted space. We can also regard $\tilde{T}(F)_{/\theta}$ as the space of $T(F)$-conjugacy classes in $\tilde{T}(F)$, therefore $D^{\tilde{G}}(\cdot)$ factors through $\tilde{T}(F)_{/\theta}$. The projection map $T_\theta(F) \to T(F)_{/\theta}$ is locally a homeomorphism. Two measures on $T_\theta(F)$, $T(F)_{/\theta}$ are called compatible if $T_\theta(F) \to T(F)_{/\theta}$ preserves these measures locally.

\begin{proposition}[{\cite[\S 1.4]{Wa09-ep}} or {\cite[(5.3.6)]{L10}}]\label{prop:Weyl-int}
  Let $f$ be an integrable continuous function on $\tilde{G}_\mathrm{reg}(F)$. Then
  \begin{multline*}
    \int_{\tilde{G}_\mathrm{reg}(F)} f(\delta) \dd\delta = \sum_{\tilde{M} \in \mathcal{L}^{\tilde{G}}(M_0)} |W^M_0| |W^G_0|^{-1} \\
    \sum_{\tilde{T}} |W(M(F),\tilde{T}(F))|^{-1} \int_{\tilde{T}(F)_{/\theta}} |D^{\tilde{G}}(t)| O^{\tilde{G}}_t(f) \dd t
  \end{multline*}
  where
  \begin{itemize}
    \item $M_0$ is a chosen minimal Levi subgroup of $G$;
    \item $\mathcal{L}^{\tilde{G}}(M_0)$ is the set of Levi subspaces $(M,\tilde{M})$ with $M \supset M_0$;
    \item $W^M_0$ is the Weyl group of $M$ with respect to $M_0$, and similarly for $W^G_0$;
    \item $(T,\tilde{T})$ runs through the $M(F)$-conjugacy classes of elliptic maximal tori of $(M,\tilde{M})$;
    \item $W(M(F),\tilde{T}(F)) := N_{M(F)}(\tilde{T}(F))/T(F)$;
    \item the measure on $\tilde{T}(F)_{/\theta}$ is induced from that on $T(F)_{/\theta}$, which is compatible with the Haar measure on $T_\theta(F)$.
  \end{itemize}
\end{proposition}
Here the orbital integral $O^{\tilde{G}}_t(f)$ is defined as in \eqref{eqn:orbint}, using the chosen measure on $T_\theta(F)$.
%\begin{remark}
%  Our choice of Haar measures is not that made in \cite{L10}. There is no contraction, as Lemaire's definition of $D^{\tilde{G}}$ is also different.
%\end{remark}

\paragraph{Stable orbital integrals}
By definition, call $\delta_1, \delta_2 \in \tilde{G}_\text{reg}(F)$ stably conjugate if they are conjugate by $G(\bar{F})$.  Call two maximal tori $(T_1, \tilde{T}_1)$, $(T_2, \tilde{T}_2)$ of $(G,\tilde{G})$ stably conjugate if there exists $x \in G(\bar{F})$ such that $\Ad_x$ induces an $F$-isomorphism $(T_1, \tilde{T}_1) \rightiso (T_2, \tilde{T}_2)$.

Choose Haar measures on $T_\theta(F)$, for every maximal tori $(T,\tilde{T})$, so that for stably conjugate $T_1$, $T_2$, the measures on $T_{1,\theta}(F)$ and $T_{2,\theta}(F)$ match. Define stable orbital integrals as follows: for each $\delta \in \tilde{G}_\text{reg}(F)$, set

\begin{gather}\label{eqn:storbint}
  SO^{\tilde{G}}_{\delta}(f) := \sum_{\delta'} O^{\tilde{G}}_{\delta'}(f), \qquad f \in C_c^\infty(\tilde{G}(F))
\end{gather}
where $\delta'$ ranges over the conjugacy classes in the stable conjugacy class of $\delta$ and $O^{\tilde{G}}_{\delta'}(f)$ is defined using the measure prescribed above. The normalized stable orbital integral is defined by
$$ S^{\tilde{G}}(\delta,f) := |D^{\tilde{G}}(\delta)|^{\frac{1}{2}} SO^{\tilde{G}}_\delta(f). $$

As before, for each maximal torus $(T,\tilde{T})$, we define $F$-varieties
\begin{align*}
  T_{/\theta} & := (1-\theta)T \backslash T, \\
  \tilde{T}_{/\theta} & := (1-\theta)T \backslash \tilde{T}.
\end{align*}
Then $(T_{/\theta}, \tilde{T}_{/\theta})$ is a twisted space. We can regard $\tilde{T}_{/\theta}(F)$ as the space of stable conjugacy classes in $\tilde{T}(F)$ under $T$-action. The projection map $T^\theta(F) \to T_{/\theta}(F)$ is locally a homeomorphism. Two measures on $T_\theta(F)$, $T_{/\theta}(F)$ are called compatible if $T_\theta(F) \to T_{/\theta}(F)$ preserves these measures locally.

\begin{proposition}\label{prop:Weyl-int-st}
    Let $f$ be an integrable continuous function on $\tilde{G}_\mathrm{reg}(F)$. Then
  \begin{multline*}
    \int_{\tilde{G}_\mathrm{reg}(F)} f(\delta) \dd\delta = \sum_{\tilde{M} \in \mathcal{L}^{\tilde{G}}(M_0)} |W^M_0| |W^G_0|^{-1} \sum_{\tilde{T}} |W(M,\tilde{T})(F)|^{-1} \\
    \int_{\tilde{T}_{/\theta}(F)} |D^{\tilde{G}}(t)| SO^{\tilde{G}}_t(f) \dd t
  \end{multline*}
  where
  \begin{itemize}
    \item $(T,\tilde{T})$ runs through the $M(F)$-stable conjugacy classes of elliptic maximal tori of $(M,\tilde{M})$;
    \item $W(M,\tilde{T}) := N_M(\tilde{T})/T$ as an $F$-group, and $W(M,\tilde{T})(F)$ is the group of its $F$-points;
    \item the measure on $\tilde{T}_{/\theta}(F)$ is induced from that on $T_{/\theta}(F)$, which is compatible with the Haar measure on $T_\theta(F)$.
  \end{itemize}
\end{proposition}
\begin{proof}
  This can be proved either by modifying the proof of Proposition \ref{prop:Weyl-int} in the stable setting, or by collecting the terms in Proposition \ref{prop:Weyl-int} according to stable conjugacy.
\end{proof}

Note that $T_\theta(F) \to T_{/\theta}(F)$ is equal to the composition $T_\theta(F) \to T(F)_{/\theta} \twoheadrightarrow T_{/\theta}(F)$; all arrows in sight are local homeomorphisms.

\begin{remark}\label{rem:untwisted-notation}
  When $(G,\tilde{G})$ is untwisted, we will remove the $\sim$ in the notations and write $O^G_\gamma(f)$, $\Theta^G_\pi(f)$, etc. (where $f \in C^\infty_c(G(F))$) in the untwisted setting, without further explanations. This conforms to Arthur's notations in \cite{ArEndo}.
\end{remark}

In the untwisted setting, strongly regular semisimple elements form a Zariski open dense subset.

\subsection{Relation to non-connected groups}\label{sec:nonconnected}
The language of non-connected reductive groups is used in other literature such as \cite{Cl87}. Recall that an $F$-group $G^+$ is called reductive if $(G^+)^\circ$ is, for example the orthogonal group $\Or(V,q)$ of an $F$-quadratic space $(V,q)$. This notion is related to twisted spaces as follows.

Let $(G,\tilde{G})$ be a twisted space and $\theta$ be the associated outer automorphism. Assume $\theta$ to be of finite order. Choose any $\delta \in \tilde{G}$, then there exist an integer $n$ and $x \in G_\text{AD}(F)$ such that $(\Ad_\delta)^n = \Ad_x$, where $G_\text{AD}$ denotes the adjoint group of $G$. Upon increasing $n$, we may even assume $x \in G(F)$. Let $\Ad_\delta^\Z$ be the free abelian group generated by the symbol $\Ad_\delta$. It acts on $G$ in the obvious manner. Define the quotient group
$$ G^+ := \left( G \rtimes \Ad_\delta^\Z \right) / \text{ the normal subgroup generated by }( x^{-1} \rtimes \Ad_\delta^n ). $$

Then $G^+$ is a non-connected reductive group, and $\tilde{G}$ can be identified as a $G$-bitorsor with $G \rtimes \Ad_\delta \subset G^+$ via $x\delta \mapsto x \rtimes \Ad_\delta$. Conversely, for any non-connected reductive group $G^+$ and a connected component $\tilde{G}$ of $G^+$, we get a twisted space $(G := (G^+)^\circ, \tilde{G})$.

By fixing the base point $\delta \in \tilde{G}$ and writing $\theta := \Ad_\delta$, we may also reduce to the $(G,\theta)$-twisted situation in \cite{KS}. Indeed, the homeomorphism $x\delta \mapsto x$ between $\tilde{G}(F)$ and $G(F)$ permits to identify $C_c^\infty(\tilde{G}(F))$ and $C_c^\infty(G(F))$. Under the identification, we have
$$ O^{\tilde{G}}_{x\delta}(f) = [G^{x\delta}(F) : G_{x\delta}(F)]^{-1} \int_{G_{x\delta}(F) \backslash G(F)} f(g^{-1} x \theta(g)) \dd g . $$

A representation $(\pi,\tilde{\pi},V)$ becomes a representation $(\pi,V)$ of $G(F)$ equipped with an intertwining operator $A: V \rightiso V$ such that $\pi\circ\theta = A \pi A^{-1}$. Indeed, up to equivalence we have $A = \tilde{\pi}(\delta)$. Its twisted character becomes the distribution $f \mapsto \Tr(\pi(f) \circ A)$ on $G(F)$.

Conversely, given a connected reductive $F$-group $G$ and $\theta \in \Aut(G)$, we can form the space $\tilde{G}$ which is isomorphic to $G$ as $F$-varieties, whose elements are written formally as $g\theta$, $g \in G$, with the $G$-bitorsor structure given by $x(g\theta)y = xg\theta(y)\theta$ for all $x,y,g \in G$. Then $(G,\tilde{G})$ is a twisted space whose outer automorphism class is exactly that of $\theta$.

\subsection{Example: $\tGL_E(N)$}\label{sec:twisted-eg}
Let $E/F$ be a field extension with $[E:F] \leq 2$. Set $\tau$ to be the nontrivial element in $\text{Gal}(E/F)$ if $[E:F]=2$, otherwise $\tau := \identity$. Let $H$ be a finite-dimensional $E$-vector space. Define
$$ \tGL_E(H) := \{q: H \times H \to E : \text{ non-degenerate } (E,\tau)-\text{sesquilinear form} \}. $$

This is an $F$-variety. Regard $\GL_E(H)$ as an $F$-group by restriction of scalars. There is a natural $\GL_E(H)$-bitorsor structure on $\tGL_E(H)$, namely for $q \in \tGL_E(H)$,
\begin{gather}
  xqy : (v,v') \mapsto q(x^{-1}v|yv), \quad x,y \in \GL_E(H).
\end{gather}

Recall that $\tGL_E(H)$ can also be identified with $\Isom_E(H,H^\vee)$: to $Y \in \Isom_E(H,H^\vee)$ we associate $q_Y: (v,v') \mapsto \angles{Yv,v'}$. The bitorsor structure can be easily described:
\begin{gather}\label{eqn:Isom-bitorsor}
  \check{y}Yx^{-1} \leftrightarrow x q_Y y, \quad x,y \in \GL_E(H).
\end{gather}

One readily checks that $(\GL_E(H),\tGL_E(H))$ forms a twisted space. In particular, the conjugacy class of $q$ is the space of $(E,\tau)$-sesquilinear forms on $H$ which are equivalent to $q$. To simplify notations, we use the same symbol $\tGL_E(H)$ to denote the set of its $F$-points. % Idem for $\GL_E(H)$

If moreover a basis $e_1, \ldots, e_N$ of $H$ over $E$ is chosen, the twisted space will also be denoted by $(\GL_E(N), \tGL_E(N))$. The outer automorphism class associated to $(\GL(N),\tGL(N))$ is nothing but that of $x \mapsto {}^* x^{-1}$, where ${}^* x$ is the hermitian transpose $x$.

When $E=F$, we write simply $(\GL(H),\tGL(H))$; this is the twisted space of non-degenerate bilinear forms on $H$. Once a basis is chosen, we will denote it by $(\GL(N), \tGL(N))$ as before.

For the reader's comfort, we record the following result.
\begin{lemma}
  In $\tGL_E(H)$, every regular semisimple element is strongly regular. In particular, strongly regular semisimple elements form a Zariski dense open subset in $\tGL_E(H)$.
\end{lemma}
\begin{proof}
  This follows from the explicit description of centralizers of semisimple elements: when $E=F$, this can be found in \cite[1.3]{Wa07}; when $[E:F]=2$, see \cite[W.II.B.2]{CHLN11}.
\end{proof}

\section{Twisted endoscopy for $\tGL(2n)$}\label{sec:twisted-endoscopy}
\subsection{$\tGL(2n)$ and its elliptic endoscopic data}\label{sec:tGL}
\paragraph{The twisted $\tGL(2n)$}
Our setting is the case ``$\nu=1$'' in \cite{Wa10}, which conforms to Arthur's setting in \cite[\S 1.2]{ArEndo}.

In this section, we consider an $F$-vector space $H$ of dimension $N=2n$. Fix a basis
$$ e_1, \ldots, e_n, e_{-n}, \ldots, e_{-1} $$
for $H$. We will study the twisted space of non-degenerate bilinear forms on $H$ defined in \S\ref{sec:twisted-eg}, denoted by $(\GL(2n),\tGL(2n))$.

We fix the base point $\tilde{\theta} \in \tGL(2n)$ defined by
$$ \tilde{\theta}(e_i|e_j) = \begin{cases}
  (-1)^i \delta_{i,-j}, & 1 \leq i \leq n,\; \forall j, \\
  (-1)^{i+1} \delta_{i,-j}, & -n \leq i \leq -1, \; \forall j,
\end{cases}
 $$
where $\delta_{i,-j}$ is Kronecker's delta. This defines a symplectic form on $F^{2n}$. Let $\theta := \Ad_{\tilde{\theta}}$. Unwinding definitions, we see that $\theta$ is the adjoint-inverse involution with respect to $\tilde{\theta}$, characterized by $\tilde{\theta}(\theta(x)^{-1}v|v') = \tilde{\theta}(v|xv')$ for all $x \in \GL(2n)$ and all $v,v'$. Moreover, $\GL(2n)^{\tilde{\theta}} = \GL(2n)_{\tilde{\theta}}$ is the symplectic group associated to $\tilde{\theta}$. We will denote the Lie algebra of $\GL(2n)_{\tilde{\theta}}$ by $\gl(2n)_{\tilde{\theta}}$.

Set
\begin{align*}
  B & := \Stab_{\GL(2n)}[ 0 \subset \angles{e_1} \subset \cdots \angles{e_1,\ldots,e_n,e_{-n}} \subset \cdots \subset \angles{e_1, \ldots, e_{-1}}], \\
  T & := \text{the diagonal subgroup of } \GL(2n).
\end{align*}
Then $(B,T)$ is a $\theta$-stable Borel pair, thus $(B,B\tilde{\theta})$ is a twisted Borel subgroup and $(T,T\tilde{\theta})$ is a twisted maximal torus. This can be seen, for example, by representing $\tilde{\theta}$ by an anti-diagonal matrix. Their $\theta$-fixed points $(B^\theta, T^\theta)$ form a Borel pair for $\GL(2n)_{\tilde{\theta}}$. These choices permit to define transfer factors unambiguously. We will return to this point in \S\ref{sec:Delta-formula1}.

\paragraph{Endoscopic data}
In this article, we only consider elliptic endoscopic data. The following definition is from \cite[p.12]{ArEndo}; see also \cite[\S 1.8]{Wa10}.

\begin{definition}
  An elliptic endoscopic datum for $\tGL(2n)$ is a triplet $(n_O, n_S, \chi)$, where
  \begin{itemize}
    \item $n_O, n_S \in \Z_{\geq 0}$ are even;
    \item $n_O + n_S = n$;
    \item $\chi: \Gamma_F \to \{\pm 1\}$ is a continuous character satisfying $\chi=1$ if $n_O=0$, and $\chi \neq 1$ if $n_O=2$.
  \end{itemize}
  To $(n_O, n_S, \chi)$ we associate the endoscopic group $G := \SO(V,q) \times \SO(n_S+1)$, where $\SO(n_S+1)$ is the split odd special orthogonal group associated to some $F$-quadratic space of dimension $n_S+1$, and $(V,q)$ is an $F$-quadratic space such that
  \begin{itemize}
    \item $\dim V = n_O$;
    \item $K = F[\sqrt{d_\pm(q)}]$ is the extension associated to $\chi$ by local class field theory (recall that $d_\pm(q) = (-1)^{n_O/2} \det q$ is the discriminant of $q$);
    \item $\SO(V,q)$ is quasisplit;
  \end{itemize}
\end{definition}
Note that $\SO(V,q) \neq \Gm$ by our conditions. Equivalently, $(V,q)$ is not isotropic of dimension two.

\begin{remark}\label{rem:bijections}
  The possible choices of $(V,q)$ can be described as follows. Write $K=F[\sqrt{\mathfrak{d}}]$ for the field extension associated to $\chi$, for some $\mathfrak{d} \in F^\times/F^{\times 2}$. We may assume $n_O > 0$, then
  $$ (V,q) \simeq (n_O-2)\Hy \oplus c\angles{1,-\mathfrak{d}} \simeq c ((n_O-2)\Hy \oplus \angles{1,-\mathfrak{d}}),$$
  where $c \in F^\times/F^{\times 2}$ is arbitrary. If $[K:F]=2$, $\angles{1,-\mathfrak{d}}$ is the quadratic form $x \mapsto N_{K/F}(x)$ on $E$; otherwise $c\angles{1,-\mathfrak{d}} \simeq \Hy$. Moreover, $K/F$ is determined by the minimal Levi subgroup of $\SO(V,q)$, which is a torus isomorphic to $(\Gm)^{n_O-2} \times \Ker N_{K/F}$ if $[K:F]=2$, otherwise split. In short, we have built up bijections
  \begin{gather}
    (n_O, n_S, \chi) \longleftrightarrow (n_O, n_S, K/F) \longleftrightarrow (n_O, n_S, \SO(V,q):\text{ quasisplit}).
  \end{gather}
\end{remark}

Let $(n_O, n_S, \chi)$ and $G = \SO(V,q) \times \SO(n_S + 1)$ as above. The general theory of twisted endoscopy furnishes a finite group $\Out_{2n}(G)$ defined in terms of automorphisms of $(n_O, n_S, \chi)$ \cite[p.18]{KS}. In our case, it is simply
$$ \Out_{2n}(G) = \Or(V,q) / \SO(V,q) = \Out(G), $$
where $\Out(\cdot)$ denotes the outer automorphism group. It is nontrivial if and only if $n_O > 0$, in which case $\Out_{2n}(G) \simeq \Z/2\Z$ is generated by any reflection of the quadratic space $(V,q)$. Hence we can regard $\Out_{2n}(G)$ as a group of $F$-automorphisms of $G$ modulo inner automorphisms coming from $G(F)$. Its actions on conjugacy classes and representations are thus well-defined.

An elliptic endoscopic datum $(n_O, n_S, \chi)$ is called simple if $n_O=0$ or $n_S=0$. Define the finite sets
\begin{align*}
  \mathcal{E}_\text{ell}(2n) & := \{ \text{elliptic endoscopic data of } \tGL(2n) \}, \\
  \mathcal{E}_\text{sim}(2n) & := \{ \text{simple endoscopic data of } \tGL(2n) \}.
\end{align*}

The general formalism of twisted endoscopy is quite involved: see \cite{KS}. Two simplifications occur in our case. Firstly, we do not need to pass to $z$-extensions of $G$ even though the derived group of $G$ is not simply connected. This is due to the existence of natural $L$-embeddings $\Lgrp{G} \hookrightarrow \Lgrp{\GL(2n)}$ for every elliptic endoscopic group $G$; see \cite[1.8]{Wa10}. Secondly, there is no need to twist the endoscopic group as in \cite{Lab04}.

For the geometric aspect, we only make use of the following concepts:
\begin{enumerate}
  \item correspondence between ``very regular'' semisimple classes -- \S\ref{sec:classes},
  \item Whittaker normalization of geometric transfer factors -- \S\ref{sec:Delta-formula1},
  \item explicit formula of transfer factor of the case $G=\SO(V,q)$, defined relative to various choices ($F$-splittings, etc.\ ) -- \S\ref{sec:Delta-formula2},
  \item transfer of orbital integrals -- \S\ref{sec:geom-trans}.
\end{enumerate}
Details will be given in the case $n_S=0$, $G=\SO(V,q)$. The spectral counterparts will be deferred to \S\ref{sec:crude-LLC}.

\subsection{Correspondence of very regular classes}\label{sec:classes}
We will use the parametrization in \cite[\S 1.3, \S 1.4]{Wa10}, paraphrased in terms of étale $F$-algebras as in \cite[\S 3]{Li11}. In each case under consideration, we are going to define a Zariski dense open subset of the subset of regular semisimple elements, called the set of very regular elements.

Besides the groups which intervene in the twisted endoscopy for $\tGL(2n)$, we also treat symplectic, unitary groups and $\tGL_E(N)$, which will be useful in \S\ref{sec:GSS-formalism}.

It should be emphasized that for $\SO(V,q)$ with even $\dim V$, we will only describe the $O(V,q)$-conjugacy classes therein.

\paragraph{Conjugacy classes in $\tGL(2n)$}
The strongly regular semisimple conjugacy classes in $\tGL(2n)$ are parametrized by triplets $(L,L_\pm,x)$, where
\begin{itemize}
  \item $L$ is an étale $F$-algebra of dimension $2n$;
  \item $L_\pm$ is the fixed subalgebra by an involution $\tau \in \Aut_F(L)$;
  \item $x \in L^\times$ generates $L$ over $F$.
\end{itemize}
Note that $\tau$ is determined by $L_\pm$, thus is omitted in the notation. This can be seen by decomposing $L_\pm$ into fields: $L_\pm = \prod_{i \in I} L_{\pm,i}$. Decompose $L = \prod_{i \in I} L_i$ accordingly, where $L_i$ is an étale $L_{\pm,i}$-algebra of dimension $2$. Hence we are reduced to the case $L_\pm$ is a field: if $L = L_\pm \times L_\pm$ then $\tau(x,y)=(y,x)$, otherwise $\tau$ is the nontrivial element in $\text{Gal}(L/L_\pm)$. This trick furnishes a dictionary between the terminology in \cite{Wa10} and ours.

To a triplet $(L,L_\pm,x)$, we choose an isomorphism of $F$-vector spaces $L \simeq F^{2n}$ and set the associated $\tilde{x} \in \tGL(2n)$ to be
$$ \tilde{x}(v|v') = \Tr_{L/F} (\tau(v)v'x), \quad v,v' \in L. $$
The conjugacy class of $\tilde{x}$ is independent of the choice of isomorphism $L \rightiso F^{2n}$. The pair of étale $F$-algebras $(L,L_\pm)$ are determined by $\tilde{x}$, up to isomorphism. The datum $x \in L^\times$, however, is determined only as an element in ${L_\pm}^\times/N_{L/L_\pm}(L^\times)$.

A strongly regular semisimple element of $\tGL(2n)$ is called very regular if its parameter $(L,L_\pm,x)$ satisfies
$$ \frac{x}{\tau(x)} \pm 1 \in L^\times .$$

\begin{remark}\label{rem:transposed-form}
  Let $\delta \in \tGL(2n)$ be parametrized by $(L,L_\pm,x)$, regarded as a non-degenerate bilinear form on $F^{}$. The following observation will be useful: if ${}^t \delta$ is the transposed form $(v,v') \mapsto \delta(v'|v)$, then $\delta + {}^t \delta$ is a quadratic form parametrized by $(L,L_\pm,x+\tau(x))$, provided that $x\tau(x)^{-1}-1 \in L^\times$. This assertion also holds in the sesquilinear context, which will be discussed later.
\end{remark}

\paragraph{Conjugacy classes in $\tGL(2n+1)$}
The strongly regular semisimple conjugacy classes in $\tGL(2n+1)$ are parametrized by quadruplets $(L,L_\pm,x,x_D)$, where $(L,L_\pm,x)$ is the same as the case of $\tGL(2n)$, and $x_D \in F^\times/F^{\times 2}$. The corresponding bilinear form $\tilde{x}$ is given by $q_x \oplus \angles{x_D}$, where $q_x$ is the bilinear form constructed from $(L,L_\pm,x)$ as before. The notion of very regular elements is unaltered; it does not depend on the extra parameter $x_D$.

\paragraph{Conjugacy classes in odd orthogonal groups}
Let $(V,q)$ be an $F$-quadratic space of dimension $2n+1$. For a semisimple element in $\SO(V,q)$, ``very regular'' means regular without eigenvalue $-1$ if regarded as an endomorphism of $V$. Conjugacy classes of such elements are parametrized by quadruplets $(L,L_\pm,y,c)$ where
\begin{itemize}
  \item $(L,L_\pm)$ and $\tau \in \Aut_F(L)$ are as above, $\dim_F L = 2n$;
  \item $y \in L^\times$ generates $L$ over $F$, and $y\tau(y)=1$;
  \item $c \in {L_\pm}^\times$, to which we associate the $F$-quadratic form on $L$ defined by
  $$ q_c(v|v') = \Tr_{L/F}(\tau(v)v' c), $$
  and we require that $(L,q_c) \oplus \angles{a} \simeq (V,q)$ for some $a \in F^\times$.
\end{itemize}
By Witt's cancellation theorem, $a$ is uniquely determined by $(L,L_\pm,c)$ up to $F^{\times 2}$. We look upon $\SO(L,q_c)$ as a subgroup of $\SO(V,q)$. To the datum $(L,L_\pm,y,c)$, we associate the element in $\SO(V,q)$ corresponding to $y \in \SO(L,q_c)$. As before, the triplet $(L,L_\pm,y)$ is uniquely determined, while $c$ is unique only in ${L_\pm}^\times/N_{L/L_\pm}(L^\times)$. This construction does not depend on the isomorphism $(L,q_c) \oplus \angles{a} \rightiso (V,q)$.

\paragraph{Conjugacy classes in even orthogonal groups}
Let $(V,q)$ be an $F$-quadratic space of dimension $2n$. For semisimple elements in $\SO(V,q)$, ``very regular'' means regular without eigenvalue $\pm 1$. Conjugacy classes by $\Or(V,q)$ of such elements are parametrized by quadruplets $(L,L_\pm,y,c)$ such that
\begin{itemize}
  \item $(L,L_\pm,y,c)$ satisfies the conditions in the odd orthogonal case, to which we associate the quadratic space $(L,q_c)$ by the same recipe;
  \item we require that $(L,q_c) \simeq (V,q)$.
\end{itemize}
Everything works in the same manner as for odd orthogonal groups, except that (1) there is no $\angles{a}$ and $\SO(L,q_c)$ can be identified with $\SO(V,q)$; (2) what we get are just the very regular $\Or(V,q)$-conjugacy classes. This suffices for our purpose. It can be shown that each very regular $\Or(V,q)$-conjugacy class splits into two $\SO(V,q)$-conjugacy classes.

\paragraph{Conjugacy classes in symplectic groups}
The symplectic case is the easiest. Let $(W,q)$ be a symplectic space of dimension $2n$, whose symplectic group is denoted by $\Sp(W,q)$. In this case, very regular means regular semisimple. Conjugacy classes in $\Sp(W,q)$ are parametrized by quadruplets $(L,L_\pm,y,c)$ such that
\begin{itemize}
  \item $(L,L_\pm,y)$ is as in the odd orthogonal case;
  \item $\dim_F L = 2n$;
  \item $c \in L^\times$ satisfies $\tau(c)=-c$.
\end{itemize}
To this datum we define a symplectic form on $L$ as $q_c: (v,v') \mapsto \Tr_{L/F} (\tau(v)v'c)$. By choosing an isomorphism of symplectic spaces $(W,q) \rightiso (L,q_c)$, we obtain the element in $\Sp(W,q)$ corresponding to $y$, whose conjugacy class is independent of the chosen isomorphism. The triplet $(L,L_\pm,y)$ is determined up to isomorphism, while $c$ is determined only up to multiplication by elements of $N_{L/L_\pm}(L^\times)$.

\paragraph{Conjugacy classes in unitary groups}
Let $E/F$ be a quadratic extension, $\tau$ be the nontrivial element in $\text{Gal}(E/F)$. Let $(V,h)$ be a $E/F$-hermitian space. Very regular elements of $\U(V,h)$ are the regular semisimple elements without eigenvalue $\pm 1$. They are parametrized by triplets $(L_\pm,y,c)$ such that
\begin{itemize}
  \item $L_\pm$ is an étale $F$-algebra;
  \item Set $L := L_\pm \otimes_F E$, which is an étale $E$-algebra, then $\tau$ induces an involution in $\Aut_{L_\pm}(L)$ which we denote by the same letter $\tau$;
  \item $y \in L^{\times}$ generates $L$ over $F$, and $y\tau(y)=1$;
  \item $c \in L_\pm$;
  \item the $E/F$-hermitian form on $L$ defined by $q_c: (v,v') \mapsto \Tr_{L/E}(\tau(v)v'c)$ is isomorphic to $(V,h)$.
\end{itemize}
To $(L_\pm,y,c)$ we associate the element corresponding to $y \in \U(L,q_c)$, whose conjugacy class is well-defined. The datum $c$ is unique in ${L_\pm}^\times/N_{L/L_\pm}(L^\times)$.

\paragraph{Conjugacy classes in $\tGL_E(N)$}
Let $(E/F,\tau)$ be as above. The strongly regular elements in $\tGL_E(N)$ are parametrized by data $(L,L_\pm,x)$, where $(L,L_\pm)$ is as in the case of unitary groups, $\dim_E L = N$, and $x \in L^\times$ is such that $F(x)=L$. The corresponding sesquilinear form is isomorphic to the form $(v,v') \mapsto \Tr_{L/E} (\tau(v)v'x)$ on $L$. The datum $x$ is determined up to multiplication by $N_{L/L_\pm}(L^\times)$. The notion of very regular elements is defined as in the case for $\tGL(2n)$.

Transposed forms ${}^t \delta$ and the associated hermitian forms ${}^t \delta + \delta$ are parametrized in the same manner as in the case of $\tGL(2n)$.

\paragraph{Centralizer and stable conjugacy}
Consider an element $\delta \in G(F)$, where $G$ is equal to one of the classical groups mentioned above; or alternatively $\delta \in \tilde{G}(F)$, where $\tilde{G}$ is the a twisted space considered in \S\ref{sec:twisted-eg}.

\begin{proposition}\label{prop:vreg-centralizer}
  If $\delta$ is very regular, then $G^\delta = G_\delta$ except for $\tilde{G}=\tGL(2n+1)$, in which case $G^\delta = G_\delta \cdot \{\pm 1\}$.
\end{proposition}
\begin{proof}
  The assertion is well-known in the untwisted case, see eg.\ \cite[\S 3]{Li11}. As for the twisted case, this follows from the description of centralizers in \cite[p.194]{Wa07} and \cite[W.II.B.2]{CHLN11}. Actually, a slightly different parametrization is used therein. Let $L_\pm = \prod_{i \in I} L_{\pm,i}$ be the decomposition into fields, and decompose $L = \prod_{i \in I} L_i$ accordingly;. write $x = (x_i)_{i \in I}$. The parameters in \cite{Wa07,CHLN11} correspond to $(x_i\tau(x_i)^{-1})_{i \in I}$. Except in the case $\tGL(2n+1)$, the possible disconnectedness of $G^\delta$ is caused by those $i$ with $x_i\tau(x_i)^{-1} = \pm 1$, which is ruled out by our assumption.
\end{proof}

\begin{proposition}\label{prop:centralizer-desc}
  Let $\delta$ be very regular with parameter $(L,L_\pm,\ldots)$, then
  $$ G_\delta \simeq L^1 := \{a \in L^\times: a\tau(a)=1 \}, $$
  viewed as an $F$-group.
\end{proposition}
\begin{proof}
  This follows from the explicit description of $G_\delta$ in \cite{Wa07,Li11,CHLN11}, as alluded above.
\end{proof}

Assume now $G$ is a quasisplit $F$-group. Two semisimple elements $\delta,\gamma \in G(F)$ are called stably conjugate if they are conjugate by $G(\bar{F})$. % Possible ramifications: no worry!

Recall that stable conjugacy is also defined for elements in $\tilde{G}_\text{reg}(F)$ in \S\ref{sec:orbint}. To get the parametrization of strongly regular semisimple classes in $\tilde{G}(F)$, where $\tilde{G}$ is one of the twisted spaces in \S\ref{sec:twisted-eg}, we replace the $x \in L^\times/N_{L/L_\pm}(L^\times)$ in a parameter $(L,L_\pm,x)$ by its image in $L^\times/L_\pm^\times$.

To get the parametrization of very regular stable conjugacy classes in classical groups (modulo $O(V,q)$-conjugacy in the even orthogonal case), it suffices to forget the datum $c$ in the parametrization above. We will simply write $(L,L_\pm,y)$, etc., for the parameters of stable classes. 

\begin{definition}[Cf.\ {\cite[\S 1.9]{Wa10}}]\label{def:correspondence}
  Let $(n_O,n_S,\chi) \in \mathcal{E}_{\text{ell}}(2n)$ and $G = \SO(V,q) \times \SO(n_S + 1)$ be the associated endoscopic group. Let $\delta \in \tGL(2n,F)$ and $\gamma = (\gamma', \gamma'') \in G(F)$ such that $\gamma'$,$\gamma''$ are very regular elements. Suppose that $\delta$ (resp. $\gamma'$, $\gamma''$) is parametrized by $(L,L_\pm,x)$ (resp. $(L', L'_\pm, y')$, $(L'', L''_\pm, y'')$. We say $\delta$ corresponds to $\gamma$, written as $\delta \leftrightarrow \gamma$, if
  \begin{enumerate}
    \item there exists an isomorphism of pairs of $F$-algebras $\varphi: (L' \times L'', L'_\pm \times L''_\pm) \rightiso (L, L_\pm)$;
    \item letting $y := \varphi(y',y'')$, we have
    $$ \frac{x}{\tau(x)} = -y. $$
  \end{enumerate}
  In this case, $\delta$ is also very regular.
\end{definition}

Clearly, the correspondence only depends on stable conjugacy classes on both side.

\subsection{A separation lemma}
In this subsection, we concentrate on simple endoscopic data of the form $(2n,0,\chi)$, with endoscopic group $G = \SO(V,q)$ where $\dim V=2n$. By Remark \ref{rem:bijections}, it suffices to investigate the quasisplit groups $\SO(V,q)$.

\begin{lemma}\label{prop:separation}
  Let $\delta \in \tGL(2n,F)$ be very regular semisimple. If $(V_1,q_1)$, $(V_2,q_2)$ are $F$-quadratic spaces of dimension $2n$ such that for each $i=1,2$,
  \begin{itemize}
    \item $\SO(V_i,q_i)$ is quasisplit;
    \item there exists a very regular $\gamma_i \in \SO(V_i,q_i)$ such that $\gamma_i \leftrightarrow \delta$ via twisted endoscopy.
  \end{itemize}
  Then there is an isomorphism $\varphi: \SO(V_1,q_1) \simeq \SO(V_2,q_2)$, and $\varphi(\gamma_1), \gamma_2$ are stably conjugate up to $\Or(V_2,q_2)$. In particular, the corresponding endoscopic data are equivalent.
\end{lemma}
\begin{proof}
  Take the parameter $(L,L_\pm,x)$ for $\delta$. For $i=1,2$, if $\gamma_i \leftrightarrow \delta$, then Definition \ref{def:correspondence} implies that the stable conjugacy class of $\gamma_i$ up to $\Or(V_i,q_i)$ is parametrized by $(L,L^\pm, y)$ with $y := -x\tau(x)^{-1}$. Hence it suffices to prove the first assertion.

  Choose $(L,L_\pm,y)$ as above. By assumption, for $i=1,2$ there exists $c_i \in L_\pm^\times$ such that
  $$ (V_i, q_i) \simeq (L, (\Tr_{L_\pm/F})_* (c_i M_{L/L_\pm})) $$
  where $c_i M_{L/L_\pm}$ is the $L_\pm$-quadratic form given by
  $$ (v,v') \mapsto \Tr_{L/L_\pm}(\tau(v)v' c_i) = c_i \Tr_{L/L_\pm}(\tau(v)v'), \quad v,v' \in L ; $$
  and $(\Tr_{L_\pm/L})_* (c_i M_{L/L_\pm})$ is its composition with $\Tr_{L_\pm/F}$, which yields an $F$-quadratic form on $L$.

  We claim that $\det q_1 = \det q_2$. Indeed, set $t := c_2 c_1^{-1}$ and write
  \begin{align*}
    c_1 M_{L/L_\pm} & = \angles{a, b}, \quad a,b \in L_\pm^\times ;\\
    c_2 M_{L/L_\pm} & = \angles{ta, tb}.
  \end{align*}
  Then we have
  \begin{align*}
    (V_1, q_1) & \simeq (\Tr_{L_\pm/F})_* \angles{a} \oplus (\Tr_{L_\pm/F})_* \angles{b}, \\
    (V_2, q_2) & \simeq (\Tr_{L_\pm/F})_* \angles{ta} \oplus (\Tr_{L_\pm/F})_* \angles{tb}.
  \end{align*}

  A straightforward calculation (cf.\ \cite[p.668]{Se84}) shows that
  $$ \det ((\Tr_{L_\pm/F})_* \angles{ta}) = N_{L_\pm/F}(t) \det((\Tr_{L_\pm/F})_* \angles{a}) $$
  and similarly for $\det((\Tr_{L_\pm/F})_* \angles{tb})$. Hence $\det q_2 = N_{L_\pm/L}(t)^2 \det q_1 = \det q_1$.

  It follows that $\SO(V_1,q_1)$ and $\SO(V_2, q_2)$ are inner forms (in fact, pure inner forms) of each other. By the uniqueness of quasisplit inner forms, we conclude that $\SO(V_1,q_1) \simeq \SO(V_2, q_2)$. % Reference: \cite[(29.29)]{KMRT98}
\end{proof}
% +converse, if needed...

\subsection{Geometric transfer factors of $G=\SO(V,q)$}\label{sec:Delta-formula1}
\paragraph{Splittings and transfer factors}
Recall that we have fixed a $\theta$-stable Borel pair $(B,T)$. Let $\Delta(B,T)$ denote the set of simple roots over $\bar{F}$ associated to $(B,T)$. A splitting of $\GL(2n)$ adapted to $(B,T)$ is a datum $(B,T, (E_\alpha)_{\alpha \in \Delta(B,T)})$, where $E_\alpha \in \gl(2n, \bar{F})$ lies in the root subspace of $\alpha$ and $E_\alpha \neq 0$. A splitting $(B,T,(E_\alpha)_\alpha)$ is called an $F$-splitting if it is $\Gamma_F$-invariant. The involution $\theta$ also acts on $F$-splittings adapted to $(B,T)$ since $(B,T)$ is $\theta$-stable.

Iterating the definition above to the Borel pair $(B^\theta, T^\theta) := (B \cap \GL(2n)_{\tilde{\theta}}, T \cap \GL(2n)_{\tilde{\theta}})$ for $\GL(2n)_{\tilde{\theta}}$, we can also define the $(B^\theta, T^\theta)$-adapted $F$-splittings $(B^\theta, T^\theta, (E_\beta)_{\beta \in \Delta(B^\theta,T^\theta)})$ of $\GL(2n)_{\tilde{\theta}}$.

In fact, splittings can be defined for all quasisplit connected reductive groups.

\begin{proposition}\label{prop:res-splitting}
  The restriction map $X^*(T) \to X^*(T^\theta)$ induces a $\Gamma_F$-equivariant bijection from the set of $\theta$-orbits of $\Delta(B,T)$ to $\Delta(B^\theta, T^\theta)$. Let $(B,T,(E_\alpha)_\alpha)$ be a $\theta$-stable splitting of $\GL(2n)$. To each $\beta \in \Delta(B^\theta, T^\theta)$ corresponding to a $\theta$-orbit $\mathcal{O}$ in $\Delta(B,T)$, put
  $$ E_\beta := \sum_{\alpha \in \mathcal{O}} E_\alpha. $$
  Then $(B,T,(E_\alpha)_\alpha) \mapsto (B^\theta, T^\theta, (E_\beta)_\beta)$ is a $\Gamma_F$-equivariant bijection from the set of $\theta$-stable $(B,T)$-adapted splittings to the set of splittings of $(B^\theta, T^\theta)$-adapted splittings. 
\end{proposition}
\begin{proof}
  The first part is contained in \cite[\S 1.3]{KS}. For the second part, see \cite[p.61]{KS}.
\end{proof}

Let $(n_O, n_S, \chi) \in \mathcal{E}_\text{ell}(2n)$ with endoscopic group $G$. Following the recipe in \cite{Wa10}, one can define a transfer factor
$$ \Delta: G_{\text{reg}}(F) \times \tGL_\text{reg}(F) \to \C $$
such that
\begin{itemize}
  \item $\Delta(\gamma,\delta) \neq 0$ if and only if $\gamma \leftrightarrow \delta$,
  \item $\Delta(\gamma,\delta)$ depends only on the conjugacy class of $\delta$ and the stable conjugacy class of $\gamma$, taken up to $\Out_{2n}(G)$.
\end{itemize}

The transfer factor also depends on the choice of an $F$-splitting $(B^\theta, T^\theta, (E_\beta)_\beta)$, or equivalently a $\theta$-stable $F$-splitting $(B,T, (E_\alpha)_\alpha)$ by Proposition \ref{prop:res-splitting}.

As we prefer to use normalized orbital integrals, our transfer factor does not contain the term $\Delta_{IV}$ in \cite{KS}.

\paragraph{Whittaker normalization}
Let $(B,T, (E_\alpha)_{\alpha \in \Delta(B,T)})$ be a $\theta$-stable $F$-splitting of $\GL(2n)$. Let $U$ be the unipotent radical of $B$. Define the algebraic homomorphism
$$\lambda_0: U \to \Ga , $$
using $(E_\alpha)_\alpha$, as in \cite[\S 3.1]{Sh10}. Recall the fixed additive character $\psi_F$ and set $\lambda := \psi_F \circ \lambda_0$ on $F$-points. Then $\lambda: U(F) \to \C^\times$ is a $\theta$-stable generic character. Conversely, every $\theta$-stable generic character of $U(F)$ arises in this way. By a $\theta$-stable Whittaker datum of $\GL(2n)$, we mean a pair $(B,\lambda)$ thus obtained.

\begin{definition}[{\cite[p.55]{ArEndo} and \cite[\S 5.3]{KS}}]\label{def:Whittaker-normalization}
  Let $(B,\lambda)$ be a $\theta$-stable Whittaker datum of $\GL(2n)$ arising from the $\theta$-stable $F$-splitting $(B,T, (E_\alpha)_{\alpha \in \Delta(B,T)})$, we define the Whittaker-normalized transfer factor for an elliptic endoscopic datum $(n_O, n_S, \chi) \in \mathcal{E}_\text{ell}(2n)$ by
  $$ \Delta_\lambda := \varepsilon\left(\frac{1}{2}, \chi, \psi_F\right)^{-1} \Delta ,$$
  where $\Delta$ is the transfer factor relative to $(B,T, (E_\alpha)_{\alpha \in \Delta(B,T)})$.
\end{definition}
It is shown in \cite[pp.65-66]{KS} that $\Delta_\lambda$ only depends on $(B,\lambda)$.

\begin{remark}\label{rem:epsilon-Weil}
  The $\varepsilon$-factor can be expressed in terms of the Weil index, as follows. If $\chi$ is trivial, set $K= F \times F$ as an $F$-algebra, otherwise $K/F$ is the quadratic extension attached to $\chi$ by local class field theory. We have
  $$ \varepsilon\left(\frac{1}{2}, \chi, \psi_F\right) = \gamma_{\psi_F}(N_{K/F}) $$
  where $N_{K/F}$ is the $F$-quadratic form $v \mapsto N_{K/F}(v)$ on $K$. We have $N_{K/F} \simeq \Hy$ if and only if $K=F \times F$.
\end{remark}

In order to use Waldspurger's formula for $\Delta$, we have to choose splittings, define and then calculate the constants $\eta, \eta_{(V,q)} \in F^\times$ in \cite[\S 1.6]{Wa10}.

\paragraph{A regular nilpotent element in $\gl(2n,F)_{\tilde{\theta}}$}
Let $(B^\theta, T^\theta, (E_\beta)_{\beta \in \Delta(B^\theta,T^\theta)})$ be an $F$-splitting of $\GL(2n)_{\tilde{\theta}}$. Set
$$ N := \sum_{\beta \in \Delta(B^\theta,T^\theta)} E_\beta . $$
Then $N$ is a regular nilpotent element in the symplectic Lie algebra $\gl(2n,F)_{\tilde{\theta}}$. Consider the symmetric bilinear form on $F^{2n}$:
$$ (v,v') \mapsto \tilde{\theta}(v|N^{2n-1}v'). $$
It is equivalent to $(\text{null form}) \oplus \angles{\eta}$ with $\eta \in F^\times/F^{\times 2}$. This is what we aim to calculate.

The construction above only depends on the $\GL(2n,F)_{\theta}$-conjugacy class of $N$. By \cite[Lemma 5.1A]{LS1} every $\GL(2n,F)_{\theta}$-conjugacy class of regular nilpotent elements in $\gl(2n,F)_{\tilde{\theta}}$ contains an element of the form $N$, for some choice of $(E_\beta)_\beta$. Therefore we can forget the $F$-splittings and concentrate on regular nilpotent elements.

Define an endomorphism $N: F^{2n} \to F^{2n}$ by
\begin{align*}
  N e_1 & = 0, \\
  N e_i & = e_{i-1}, \quad  1 < i \leq n, \\
  N e_{-n} & = e_n , \\
  N e_i & = e_{i-1}, \quad -n < i \leq -1. 
\end{align*}
One can check that $N \in \gl(2n,F)_{\tilde{\theta}}$. Moreover, $N$ is regular nilpotent and
$$ N^{2n-1}e_i =
  \begin{cases}
    e_1, & \text{if } i=-1, \\
    0, & \text{otherwise}.
  \end{cases}$$
Hence $\tilde{\theta}(e_{-1}|N^{2n-1}e_{-1}) = \tilde{\theta}(e_{-1}|e_1)=1$. Summing up:

\begin{lemma}\label{prop:eta-sp}
  Suppose the $F$-splitting $(B^\theta, T^\theta, (E_\beta)_{\beta \in \Delta(B^\theta,T^\theta)})$ is chosen so that the associated regular nilpotent element is conjugate to the $N$ defined above. Then $\eta=1$.
\end{lemma}

Our choice of $N$ is compatible with the $\theta$-stable Whittaker datum $(B(2n),\lambda(2n))$ of $\GL(n)$ chosen in \cite[p.55]{ArEndo}.

\paragraph{A regular nilpotent element in $\so(V,q)$}
In this paragraph we consider an $F$-quadratic space $(V,q)$ of dimension $2n$ with quasisplit $\SO(V,q)$.

For $n=1$, define $\eta_{(V,q)} \in F^\times/F^{\times 2}$ to be any class represented by $q$.

Henceforth we assume $n>1$. Let $(B',T', (E_\alpha)_{\alpha \in \Delta(B',T')})$ be an $F$-splitting of $\SO(V,q)$. Set
$$ N := \sum_{\alpha \in \Delta(B',T')} E_\alpha $$
as usual, then $N$ is a regular nilpotent element in $\so(V,q)$. Consider the symmetric bilinear form on $V$:
$$ (v,v') \mapsto q(v|N^{2n-2}v'). $$
As in the previous case, it is equivalent to $\angles{\eta_{(V,q)}}$ for some $\eta_{(V,q)} \in F^\times/F^{\times 2}$ modulo null forms. This defines $\eta_{(V,q)}$. As before, we forget $F$-splittings and just consider any regular nilpotent $N \in \so(V,q)$.

To study $\eta_{(V,q)}$, the first step is to reduce to the split odd orthogonal case. Set $m := n-1$ and write
\begin{gather}\label{eqn:ani-kernel}
  (V,q) \simeq m\Hy \oplus (V', q')
\end{gather}
where $(V',q')$ is a $2$-dimensional $F$-quadratic space, uniquely determined by Witt's cancellation theorem. In fact, $(V',q') \simeq \Hy$ if $\SO(V,q)$ is split, otherwise it is the anisotropic kernel of $(V,q)$.

Let $y \in F^\times$ be any element represented by $q'$. There exists a unique $y' \in F^\times/F^{\times 2}$ such that $(V',q') \simeq \angles{y,y'}$. We set $(V^\flat, q^\flat) := m\Hy \oplus \angles{y}$ so that
$$ (V,q) \simeq (V^\flat, q^\flat) \oplus \angles{y'}, $$
and $\SO(V^\flat, q^\flat)$ is a split odd orthogonal group.

There is only one regular nilpotent class in $\so(V^\flat,q^\flat)$; take $N$ to be any element therein. Upon extension by zero, it is regarded as a nilpotent element in $\so(V,q)$. By \cite[\S 1]{Wa10-GP} we know $N$ is regular nilpotent in $\so(V,q)$, and any regular nilpotent element arises in this manner.

Now begins the construction of $N$. Choose a basis $e_{\pm 1}, \ldots, e_{\pm m}$ of $m\Hy$ such that
$$ q^\flat(e_i|e_j) = \delta_{i,-j}, \quad -m \leq i,j \leq m .$$

Choose $v$ in the underlying space of $\angles{y}$ such that $q^\flat(v)=y$. Set
\begin{align*}
  N e_i & = e_{i+1}, \quad 1 \leq i < m, \\
  N e_m & =v, \\
  N v & = -y e_{-m}, \\
  N e_i & = -e_{i+1}, \quad -m \leq i < -1, \\
  N e_{-1} & = 0.
\end{align*}
One checks that $N$ is regular nilpotent in $\so(V^\flat,q^\flat)$. A simple calculation yields $N^{2m} e_1 = (-1)^m y e_{-1}$ and $N^{2m}x=0$ for every other element $x$ in the basis of $V^\flat$. Hence
\begin{gather}
  \eta_{(V,q)} := q^\flat(e_1|N^{2n-2}e_1) = (-1)^{n-1} y \cdot q^\flat(e_1|e_{-1}) = (-1)^{n-1} y.
\end{gather}

\begin{lemma}\label{prop:eta-so}
  With the notations above, $\eta_{(V,q)}$ can be $(-1)^{n-1}$ times any nonzero element represented by $(V',q')$.
\end{lemma}
\begin{proof}
  This has just been done for $n>1$. When $n=1$, we have $(V',q')=(V,q)$ and the assertion is true by definition.
\end{proof}

\paragraph{The formula for $G=\SO(V,q)$}
The crucial tool in this article is the following formula of Waldspurger. Let $(V,q)$ be an $F$-quadratic space of dimension $2n$ such that $\SO(V,q)$ is quasisplit. We associate the simple endoscopic datum $(2n,0,\chi)$ of $\tGL(2n)$ to $\SO(V,q)$ as in Remark \ref{rem:bijections}. Fix the $F$-splitting of $\GL(2n)_{\tilde{\theta}}$ chosen in Lemma \ref{prop:eta-sp}. Hence one can unambiguously talk about correspondence of regular semisimple classes and the transfer factor. 

\begin{theorem}\label{prop:wald-evenso}
  Let $\delta \in \tGL(2n)$, regarded as a non-degenerate bilinear form on $F^{2n}$, and $\gamma \in \SO(V,q)$ such that $\gamma$ is very regular and $\gamma \leftrightarrow \delta$. Define a quadratic form $q_\delta$ on $F^{2n}$ by
  $$ q_\delta := \frac{1}{2} ( \delta + {}^t \delta ) = \left[ (v,v') \mapsto  \frac{1}{2} (\delta(v|v') + \delta(v'|v)) \right]. $$

  If there exists an $F$-splitting for $\SO(V,q)$ such that $\eta_{(V,q)}=-1$, then
  $$ \Delta(\gamma,\delta) = \begin{cases}
    1, & \text{if } (F^{2n}, q_\delta) \simeq (V,q), \\
    -1, & \text{otherwise}.
  \end{cases}$$
\end{theorem}
\begin{proof}
  In view of Lemma \ref{prop:eta-sp}, this is an immediate consequence of \cite[\S 1.11 (2)]{Wa10}. The factor $\frac{1}{2}$ is missing in \cite{Wa10}.
\end{proof}

Note that the non-degeneracy of $q_\delta$ follows from the fact that $\delta$ is very regular; see Definition \ref{def:correspondence} and Remark \ref{rem:transposed-form}.

To get rid of the dependence on $(V,q)$ and the $F$-splitting of $\SO(V,q)$, we write $(V,q) \simeq (n-1)\Hy \oplus (V',q')$ as in \eqref{eqn:ani-kernel}. Write
$$ (V',q') = c N_{K/F} $$
where
\begin{itemize}
  \item $c \in F^\times/F^{\times 2}$;
  \item $K$ is either the étale $F$-algebra $F \times F$, or a quadratic extension of $F$; in either case, $\tau$ denotes the nontrivial element in $\Aut_F(K)$ and $N_{K/F}$ is the norm form on $K$, characterized by
  $$ v \mapsto N_{K/F}(v) = v\tau(v), \quad v \in K. $$
\end{itemize}
Note that $K$ is uniquely determined by $\SO(V,q)$. When $K=F \times F$, we have $N_{K/F} \simeq cN_{K/F} \simeq \Hy$.

Let $(V_1,q_1)$, $(V_2,q_2)$ be two $F$-quadratic spaces. We write $(V_1,q_1) \stackrel{\text{Witt}}{\sim} (V_2,q_2)$ if they have the same image in the Witt group over $F$.

\begin{corollary}\label{prop:Delta-formula-gen}
  For $\gamma \leftrightarrow \delta$ as before, we have
  $$\Delta(\gamma,\delta) = \begin{cases}
    1, & \text{if } (F^{2n},q_\delta) \stackrel{\mathrm{Witt}}{\sim} (-1)^n N_{K/F}, \\
    -1, & \text{otherwise}.
  \end{cases}$$
\end{corollary}
\begin{proof}
  For any $F$-splitting of $\SO(V,q)$ with associated $\eta_{(V,q)}$, we take $t := -\eta_{(V,q)}$. Using the same $F$-splitting, we get
  $$ \eta_{(V,tq)} = t\eta_{(V,q)}=-1.$$
  Thus Theorem \ref{prop:wald-evenso} is applicable. It remains to interpret the condition $q_\delta \simeq (V,tq)$. As $(V,q) \stackrel{\text{Witt}}{\sim} cN_{K/F} = (V',q')$, the elements in $F^\times$ represented by $(V',q')$ is $cN_{K/F}(K^\times)$ (this holds trivially when $K=F \times F$). Hence for any choice of $F$-splitting of $\SO(V,q)$, we have
  \begin{align*}
    (V,tq) \stackrel{\text{Witt}}{\sim} tcN_{K/F} & = -(-1)^{n-1} c^2 N_{K/F}(K^\times) \cdot N_{K/F} \\
    & = (-1)^n N_{K/F}
  \end{align*}
  by using Lemma \ref{prop:eta-so}.

  Now Witt's theorems imply that $q_\delta \simeq (V,tq)$ if and only if $q_\delta \stackrel{\text{Witt}}{\sim} (-1)^n N_{K/F}$, as required.
\end{proof}

\subsection{Geometric transfer}\label{sec:geom-trans}
The constructions below are due to Arthur \cite{Ar96,ArEndo}.
\paragraph{The twisted side}
Define
\begin{align*}
  \Gamma_\text{reg}(\tGL(2n)) & := \{ \text{strongly regular semisimple conjugacy classes in } \tGL(2n,F) \}, \\
  \Gamma_\text{reg,ell}(\tGL(2n)) & := \{ \delta \in \Gamma_\text{reg}(\tGL(2n,F)) : \delta \text{ is elliptic} \}.
\end{align*}

Choose Haar measures on $\GL(2n,F)$ and on the centralizers $\tGL(2n,F)_\delta$ for each $\delta \in \Gamma_\text{reg}(\tGL(2n))$, so that the orbital integrals are well-defined. Set $C_0^\infty(\tGL(2n,F))$ to be the vector space of functions $f \in C_c^\infty(\tGL(2n,F))$ such that $I^{\tGL(2n)}(\delta,f)=0$ for all $\delta \in \Gamma_\text{reg}(\tGL(2n))$. Define
\begin{gather}
  \mathcal{I}(\tGL(2n)) := C_c^\infty(\tGL(2n,F))/C_0^\infty(\tGL(2n,F)).
\end{gather}
Elements in $\mathcal{I}(\tGL(2n))$ can be seen as as functions on $\Gamma_\text{reg}(\tGL(2n))$: $\delta \mapsto I^{\tGL(2n)}(\delta,f)$, where $f$ is fixed. It can also be viewed as the quotient of $C_c^\infty(\tGL(2n,F))$ by the subspace spanned by functions $f^y - f$, where $f$ ranges over $C_c^\infty(\tGL(2n,F))$ and $f^y(\delta) := f(y\delta y^{-1})$, $y \in \GL(2n,F)$. This follows from the twisted version of the density of semisimple regular orbital integrals. To prove this, it suffices to reduce to the corresponding result on Lie algebras \cite[Lemma 4.1]{HC99} using twisted descent \cite[2.4]{Wa08}; see also \cite[Proposition 4.1.5]{Li11-FT2}.

A function $f \in C_c^\infty(\tGL(2n,F))$ is called cuspidal if $I^{\tGL(2n)}(\delta,f)=0$ for all non-elliptic $\delta \in \Gamma_\text{reg}(\tGL(2n))$. This notion only depends on the image of $f$ in $\mathcal{I}(\tGL(2n))$. Define
\begin{gather}
  \mathcal{I}_\text{cusp}(\tGL(2n)) := \{f \in \mathcal{I}(\tGL(2n)) : f \text{ is cuspidal} \}.
\end{gather}

\paragraph{The endoscopic side}
For any $(n_O,n_S,\chi) \in \mathcal{E}_\text{ell}(2n)$ with endoscopic group $G$, we set
\begin{align*}
  \Gamma_\text{reg}(G) & := \{ \text{very regular semisimple conjugacy classes in } G(F) \},\\
  \Gamma_\text{reg,ell}(G) & := \{ \gamma \in \Gamma_\text{reg}(G) : \gamma \text{ is elliptic} \}, \\
  \Delta_\text{reg}(G) & := \Gamma_\text{reg}(F)/\text{stable conjugacy}, \\
  \Delta_\text{reg,ell}(G) & := \Gamma_\text{reg,ell}(F)/\text{stable conjugacy}.
\end{align*}
Recall that very regular elements have connected centralizer by Proposition \ref{prop:vreg-centralizer}. To define stable orbital integrals, we choose Haar measures on $G(F)$ and on the maximal tori in $G$, such that if $\gamma_1, \gamma_2 \in \Delta_\text{reg}(G)$ are stably conjugate up to $\Out_{2n}(G)$, then the isomorphism $G_{\gamma_1}(F) \simeq G_{\gamma_2}(F)$ preserves Haar measures. Such choices are possible. For $f \in C_c^\infty(G(F))$ and $\gamma \in \Delta_\text{reg}(G)$, we can define the normalized stable orbital integral
$$ S^G(\gamma,f) := \sum_{\gamma'} I^G(\gamma',f) $$
where $\gamma'$ ranges over the conjugacy classes in the stable conjugacy class $\gamma$.

As before, set $C_0^\infty(G(F))$ to be the vector space of functions $f \in C_c^\infty(G(F)))$ such that $I^G(\gamma,f)=0$ for all $\gamma \in \Gamma_\text{reg}(G)$ and define $\mathcal{I}(G) := C_c^\infty(G(F))/C_0^\infty(G(F))$. There is a stable variant, namely set $C_1^\infty(G(F)) \subset C_c^\infty(G(F))$ to be the subspace of those $f$ satisfying $S^G(\gamma,f)=0$ for all $\delta \in \Delta_\text{reg}(G)$; define
\begin{gather}
  S\mathcal{I}(G) := C_c^\infty(G(F))/C_1^\infty(G(F)).
\end{gather}

There is a surjection $\mathcal{I}(G) \twoheadrightarrow S\mathcal{I}(G)$. The algebraic dual of $S\mathcal{I}(G)$, regarded as a space of invariant distributions on $G(F)$, is by definition the space of stable distributions on $G(F)$. Elements in $S\mathcal{I}(G)$ can be seen as functions on $\Delta_\text{reg}(G)$: $\gamma \mapsto S^G(\gamma,f)$; as in the previous case, an element in $S\mathcal{I}(G)$ is called cuspidal if its restriction on non-elliptic stable classes is identically zero. Define
\begin{gather}
  S\mathcal{I}_\text{cusp}(G) := \{f \in S\mathcal{I}(G) : f \text{ is cuspidal} \}.
\end{gather}

The group $\Out_{2n}(G)$ acts on conjugacy classes, therefore one can define
\begin{align*}
  \overline{\Gamma}_\text{reg}(G) & := \Gamma_\text{reg}(G)/\Out_{2n}(G), &
  \overline{\Gamma}_\text{reg,ell}(G) & := \Gamma_\text{reg,ell}(G)/\Out_{2n}(G), \\
  \overline{\Delta}_\text{reg}(G) & := \Delta_\text{reg}(G)/\Out_{2n}(G), &
  \overline{\Delta}_\text{reg,ell}(G) & := \Delta_\text{reg,ell}(G)/\Out_{2n}(G),
\end{align*}
\begin{align*}
  \overline{\mathcal{I}}(G) & := \mathcal{I}(G)^{\Out_{2n}(G)},  \\
  \overline{S\mathcal{I}}(G) & := S\mathcal{I}(G)^{\Out_{2n}(G)}, \\
  \overline{S\mathcal{I}}_\text{cusp}(G) & := S\mathcal{I}_\text{cusp}(G)^{\Out_{2n}(G)}.
\end{align*}
The elements in $\overline{S\mathcal{I}}(G)$ can thus be viewed as functions on $\overline{\Delta}_\text{reg}(G)$.

\paragraph{Compatibility of measures}
Summing up, we have chosen Haar measures for the groups below:
\begin{itemize}
  \item $\GL(2n,F)$, $G(F)$;
  \item $T_\theta(F)$, $\tilde{T}(F)_{/\theta}$, $\tilde{T}_{/\theta}(F)$, for every maximal torus $(T,\tilde{T})$ of $\tGL(2n)$;
  \item $T_G(F)$, for every maximal torus $T_G$ of $G$.
\end{itemize}
Here $G$ ranges over all elliptic endoscopic data of $\tGL(2n)$.

\begin{definition}\label{def:measures}
  The Haar measures above are said to be compatible if the following conditions hold.
  \begin{enumerate}
    \item The Haar measures on $T_\theta(F)$, $\tilde{T}(F)_{/\theta}$, $\tilde{T}_{/\theta}(F)$ are compatible in the sense of \S\ref{sec:orbint}. Moreover, they are compatible with stable conjugacy of maximal tori in $\tGL(2n)$.
    \item The Haar measures on $T_G(F)$ are compatible with stable conjugacy of maximal tori in $G$.
    \item Suppose $\gamma \in \Delta_\text{reg}(G)$, $\delta \in \Gamma_\text{reg}(\tGL(2n))$, $\gamma \leftrightarrow \delta$. Set $T_G := G_\gamma$ and $\tilde{T} := Z_{\tGL(2n)}(\GL(2n)_\delta)$, then $T_\theta = \GL(2n)_\delta$ and twisted endoscopy provides an isomorphism $T_{/\theta} \rightiso T_G$. We assume that the Haar measures on $T_G(F)$ and $T_\theta(F)$ match under these identifications.
  \end{enumerate}
%We have the exact sequences
%    \begin{gather*}
%      1 \to T^\theta \to T \xrightarrow{1-\theta} T \to T_{/\theta} \to 1 , \\
%      0 \to \mathfrak{t}_\theta \to \mathfrak{t} \xrightarrow{1-\theta} \mathfrak{t} \to \mathfrak{t}_{/\theta} \to 0 ,
%    \end{gather*}
%    which induces a canonical isomorphism $\bigwedge^\text{max} \mathfrak{t}^*_\theta \rightiso \bigwedge^\text{max} \mathfrak{t}^*_{/\theta}$.
\end{definition}
This is the convention in \cite[3.10]{Wa08}. It differs from that in \cite[5.5]{KS} by a harmless constant.

\paragraph{Langlands-Shelstad-Kottwitz transfer}
Assume that a $\theta$-stable $F$-splitting of $\tGL(2n)$ is chosen so that the associated $F$-splitting of $\GL(2n)_{\tilde{\theta}}$ is as in Lemma \ref{prop:eta-sp}. From this we fabricate a $\theta$-stable Whittaker datum $(B,\lambda)$ and the Whittaker-normalized transfer factor $\Delta_\lambda$ for each $G$ (see Definition \ref{def:Whittaker-normalization}).

\begin{theorem}[Ngô \cite{Ng08}]\label{prop:LSK}
  For every $(n_O,n_S,\chi) \in \mathcal{E}_\mathrm{ell}(2n)$ with endoscopic group $G$, there exists a map
  \begin{align*}
    \mathcal{I}(\tGL(2n)) & \longrightarrow \overline{S\mathcal{I}}(G) \\
    f & \longmapsto f^G,
  \end{align*}
  characterized by the equation
  \begin{gather}\label{eqn:LSK-char-0}
    S^G(\gamma, f^G) = \sum_{\substack{\delta \in \Gamma_{\mathrm{reg}}(\tGL(2n)) \\ \gamma \leftrightarrow \delta}} \Delta_\lambda(\gamma,\delta) I^{\tGL(2n)}(\delta, f), \quad \gamma \in \Delta_\mathrm{reg}(G).
  \end{gather}
\end{theorem}
Here the orbital integrals are defined using compatible Haar measures in Definition \ref{def:measures}.

By using this theorem, we look upon $\gamma \mapsto S^G(\gamma, f^G)$ as functions on $\overline{\Delta}_\mathrm{reg}(G)$.

Varying endoscopic data, we set
\begin{align*}
  \Gamma_{\text{reg,ell}}^{\mathcal{E}}(\tGL(2n)) & := \bigsqcup_G \overline{\Delta}_{2n-\text{reg,ell}}(G), \\
  \mathcal{I}_{\text{cusp}}^{\mathcal{E}}(\tGL(n)) & := \bigoplus_G \overline{S\mathcal{I}}_\text{cusp}(G).
\end{align*}
where
\begin{itemize}
  \item by abuse of notation, $G$ ranges over elements in $\mathcal{E}_\text{ell}(2n)$ for which $G$ is the associated endoscopic group;
  \item $\overline{\Delta}_{2n-\text{reg,ell}}(G)$ signifies the set of elements in $\overline{\Delta}_\text{reg,ell}(G)$ that correspond to some strongly regular semisimple element in $\tGL(2n,F)$.
\end{itemize}

The correspondence $\gamma \leftrightarrow \delta$ induces a correspondence between $\Gamma_{\text{reg,ell}}^{\mathcal{E}}(\tGL(2n))$ and $\Gamma_{\text{reg,ell}}(2n)$, since it preserves $F$-ellipticity \cite[5.5]{KS}. The Whittaker-normalized transfer factors of various $G$ also merge into a single factor
$$ \Delta_\lambda: \Gamma_{\text{reg,ell}}^{\mathcal{E}}(\tGL(2n)) \times \Gamma_{\text{reg,ell}}(\tGL(2n)) \to \C $$
such that $\Delta_\lambda(\gamma,\delta) \neq 0$ if and only if $\gamma \leftrightarrow \delta$.

\begin{theorem}\label{prop:transfer-surj}
  The transfer maps $f \mapsto f^G$ induces an isomorphism
  \begin{align*}
    \mathcal{I}_\mathrm{cusp}(\tGL(2n)) & \stackrel{\sim}{\longrightarrow} \mathcal{I}_{\mathrm{cusp}}^{\mathcal{E}}(\tGL(n)) \\
    f & \longmapsto f' := (f^G)_G.
  \end{align*}
\end{theorem}
\begin{proof}
  This is ``the second crucial step'' in the proof of \cite[Proposition 2.1.1]{ArEndo}, pp.58-59.
\end{proof}

Rewriting \eqref{eqn:LSK-char-0}, the isomorphism $f \mapsto f'$ is characterized as follows
\begin{gather}\label{eqn:LSK-char}
  S(\gamma, f') = \sum_{\substack{\delta \in \Gamma_{\text{reg}}(\tGL(2n)) \\ \gamma \leftrightarrow \delta}} \Delta_\lambda(\gamma,\delta) I^{\tGL(2n)}(\delta, f), \quad \gamma \in \Gamma_{\text{reg,ell}}^{\mathcal{E}}(\tGL(2n)), 
\end{gather}
where we define $S(\gamma, f') := S^G(\gamma, f^G)$ if $\gamma \in \overline{\Delta}_{\text{reg,ell}}(G)$.

\section{Crude local Langlands correspondence for orthogonal groups}\label{sec:crude-LLC}
\subsection{Selfdual $L$-parameters}
This is mainly a review of well-known results in order to fix notations. Let $N \in \Z_{\geq 1}$. Recall that $\WD_F := \text{W}_F \times \SU(2)$.

\begin{definition}
  A $L$-parameter for $\GL(N,F)$ is a semisimple, continuous representation
  $$ \phi: \WD_F \to \GL(N,\C) $$
  taken up to equivalence, i.e.\ up to $\GL(N,\C)$-conjugacy.
\end{definition}
%\begin{remark}
%  In \cite{ArEndo}, Arthur considers parameters of the form $\phi: \WD_F \times \SU(2) \to \GL(N,\C)$. Our parameters correspond to what he called generic parameters, that is, the parameters with $\phi|_{1 \times \SU(2)}=1$.
%\end{remark}

Introduce the nested family of $L$-parameters as follows.
\begin{align*}
  \Phi(\GL(N)) & := \{ L-\text{parameters of } \GL(N,F) \}, \\
  \Phi_\text{bdd}(\GL(N)) & := \{ \phi \in \Phi(\GL(N)) : \phi \text{ has bounded image in } \GL(N,\C) \}, \\
  \Phi_{2,\text{bdd}}(\GL(N)) & := \{ \phi \in \Phi_\text{bdd}(\GL(N)) : \text{irreducible} \}, \\
  \Phi_{\text{sc,bdd}}(\GL(N)) & := \{ \phi \in \Phi_{2,\text{bdd}}(\GL(N)) : \phi|_{1 \times \SU(2)} = 1 \}.
\end{align*}
Denote the contragredient operation of $L$-parameters for $\GL(N,F)$ by $\phi \mapsto \phi^\vee$. It preserves each of the subsets above.

On the other hand, define
\begin{align*}
  \Pi(\GL(N)) & := \{ \text{smooth irreducible representations of } \GL(N,F) \}/\sim, \\
  \Pi_{\text{temp}}(\GL(N)) & := \{ \pi \in \Pi(\GL(N)): \pi \text{ is tempered} \}, \\
  \Pi_{2,\text{temp}}(\GL(N)) & := \{ \pi \in \Pi_{\text{temp}}(\GL(N)) : \pi \text{ is essentially square-integrable} \}, \\
  \Pi_{\text{sc,temp}}(\GL(N)) & := \{ \pi \in \Pi_{2,\text{temp}}(\GL(N)) : \pi \text{ is supercuspidal} \}.
\end{align*}
Recall that a tempered irreducible representation is unitary.

Now we can state a small portion of local Langlands correspondence for $\GL(N,F)$.
\begin{theorem}[{\cite{HT01,He00}}]
  There is a canonical bijection between $\Phi(\GL(N))$ and $\Pi(\GL(N))$, written as $\phi \leftrightarrow \pi$, such that
  \begin{enumerate}
    \item when $N=1$, the bijection is given by local class field theory; %(to be precise, we require that the reciprocity map sends geometric Frobenius to prime elements in $F^\times$);
    \item for $\phi \leftrightarrow \pi$, the central character $\omega_\pi: F^\times \to \C^\times$ of $\pi$ corresponds to $\det \circ \phi: \WD_F \to \C^\times$;
    \item $\phi^\vee \leftrightarrow \pi^\vee$ if and only if $\phi \leftrightarrow \pi$;
    \item the correspondence matches
      \begin{eqnarray*}
        \Phi_{\mathrm{bdd}}(\GL(N)) & \longleftrightarrow & \Pi_{\mathrm{temp}}(\GL(N)) \\
        \Phi_{2,\mathrm{bdd}}(\GL(N)) & \longleftrightarrow & \Pi_{2,\mathrm{temp}}(\GL(N)) \\
        \Phi_{\mathrm{sc,bdd}}(\GL(N)) & \longleftrightarrow & \Pi_{\mathrm{sc,temp}}(\GL(N)).
      \end{eqnarray*}
  \end{enumerate}
\end{theorem}
\begin{remark}
  The correspondence can be characterized by further properties, namely the matching for Rankin-Selberg $L$-functions and $\varepsilon$-factors; see \cite{He02} for details.
\end{remark}

Let $\phi \in \Phi(\GL(N))$. Call $\phi$ selfdual if $\phi \simeq \phi^\vee$. Selfdual $\phi$ have an unique decomposition into subrepresentations of $\WD_F$:
\begin{gather}\label{eqn:L-decomposition}
  \phi = \bigoplus_{i \in I_\phi} \ell_i \phi_i \oplus \bigoplus_{j \in J_\phi} \ell_j (\phi_j \oplus \phi_j^\vee),
\end{gather}
where
\begin{itemize}
  \item $\ell_i, \ell_j \in \Z_{\geq 1}$;
  \item the representations $(\phi_i)_{i \in I_\phi}$ are irreducible and distinct, idem for $(\phi_j)_{j \in J_\phi}$;
  \item $\phi_i^\vee \simeq \phi_i$ for all $i \in I_\phi$;
  \item $\phi_j^\vee \not\simeq \phi_j$ for all $j \in J_\phi$;
\end{itemize}

For each $i \in I_\phi$ as above, there exists an isomorphism $f_i: \phi_i \rightiso \phi_i^\vee$. By identifying $\phi_i^{\vee\vee} = \phi_i$ and consequently $f_i^{\vee\vee}=f_i$, we see that there exists a well-defined sign $\sgn(\phi_i) = \pm 1$ such that $f_i^\vee = \sgn(\phi_i) f_i$ for every choice of $f_i$.

We write
\begin{align*}
  \Phi_\text{bdd}(\tGL(N)) & := \{ \phi \in \Phi_\text{bdd}(\GL(N)) : \phi \simeq \phi^\vee \}, \\
  \Phi_\text{ell,bdd}(\tGL(N)) & := \{ \phi \in \Phi_\text{bdd}(\tGL(N)) : J_\phi=\emptyset,\; \forall i \in I_\phi,\; \ell_i=1 \text{ in } \eqref{eqn:L-decomposition} \}.
\end{align*}

\begin{remark}\label{rem:twisted-pi}
  To justify our notation, let us show how to associate a representation $(\pi,\tilde{\pi},V)$ of $\tGL(N,F)$ to a $\phi \in \Phi_\text{bdd}(\tGL(N))$. Fix an element $\delta \in \tGL(N,F)$ such that $\theta := \Ad_\delta$ fixes a Whittaker datum $(B,\lambda)$. We have the notion of $(B,\lambda)$-generic representations \cite[Definition 3.1.2]{Sh10}. Take $(\pi,V) \in \Pi_\text{temp}(\GL(N))$ such that $\phi \leftrightarrow \pi$, then $(\pi, V)$ is selfdual. Jacquet's theorem \cite{Ja77} asserts that tempered representations of $\GL(N,F)$ are $(B,\lambda)$-generic. Hence we can define $\tilde{\pi}$ by requiring that $A := \tilde{\pi}(\delta)$ satisfies $\pi \circ \theta = A \pi A^{-1}$ and $\omega \circ A = \omega$ for every $(B,\lambda)$-Whittaker functional $\omega$ of $\pi$. Conversely, every strongly irreducible representation $(\pi,\tilde{\pi},V)$ of $\tGL(N,F)$ with tempered $\pi$ is so obtained up to equivalence.
\end{remark}

\subsection{Spectral transfer}\label{sec:spectral-transfer}
\paragraph{The setting}
Henceforth we specialize to $N=2n$ and study the parameters $\phi \in \Phi_\text{ell,bdd}(\tGL(2n))$.

Write $\phi = \bigoplus_{i \in I_\phi} \phi_i$ as in \eqref{eqn:L-decomposition}. Every $\phi_i$ is viewed as an element in $\Phi_\text{ell,bdd}(\tGL(n_i))$ for some $n_i$. Set $\chi_i := \det \circ \phi_i \in \Phi_\text{bdd}(\GL(1))$, it can be regarded as a continuous homomorphism $\Gamma_F \to \{\pm 1\}$ by the self-duality of $\phi_i$. Note that $\sgn(\phi_i)=-1$ implies $\chi_i=1$.

Consider an elliptic endoscopic datum $(n_O,n_S,\chi) \in \mathcal{E}_\text{ell}(2n)$ with endoscopic group $G = \SO(V,q) \times \SO(n_S+1)$. Define
\begin{align*}
  \Pi(G) & := \{ \pi: \pi \text{ is an irreducible smooth representation of } G(F) \} / \sim, \\
  \Pi_2(G) & := \{ \pi \in \Pi(G) : \pi \text{ is square-integrable} \}, \\
  \overline{\Pi}(G) & := \Pi(G) / \Out(G), \\
  \overline{\Pi_2}(G) & := \Pi_2(G) / \Out(G).
\end{align*}

In \S\ref{sec:geom-trans} we have defined the spaces $\mathcal{I}(G)$, $\overline{\mathcal{I}}(G)$. Since $\Out(G)=\Out_{2n}(G)$, we have $\overline{\mathcal{I}}(G) = \mathcal{I}(G)^{\Out(G)}$. Fix a Haar measure on $G(F)$. Each $\bar{\sigma} \in \overline{\Pi}(G)$ defines a linear functional
\begin{align*}
  \Theta^G_\sigma : \overline{\mathcal{I}}(G) & \longrightarrow \C \\
  f & \longmapsto \Theta^G_\sigma(f) = \Tr \left(\;\int_{G(F)} f(x) \pi(x) \dd x \right),
\end{align*}
where $\sigma$ is any inverse image of $\bar{\sigma}$ in $\Pi(G)$.

Now we define the $L$-parameters corresponding to elements in $\overline{\Pi}_2(G)$. Let $\phi \in \Phi_\text{ell,bdd}(\tGL(2n))$ with the decomposition \eqref{eqn:L-decomposition}. Define
\begin{align*}
  I_\phi^+ & := \{i \in I_\phi: \sgn(\phi_i)=1 \},\\
  I_\phi^- & := \{i \in I_\phi: \sgn(\phi_i)=-1 \}.
\end{align*}

Recall the Remark \ref{rem:bijections} that $\chi$ is a continuous character $\Gamma_F \to \{\pm 1\}$ which determines $\SO(V,q)$. Define
\begin{gather}\label{eqn:Phi_2(G)}
  \overline{\Phi_2}(G) := \left\{\phi \in \Phi_\text{ell,bdd}(\tGL(2n)) : |I_\phi^-|=n_S, \; \prod_{i \in I_\phi^+} \chi_i =\chi \right\}.
\end{gather}
Therefore
\begin{gather}\label{eqn:Phi_2-disjoint}
  \Phi_\text{ell,bdd}(\tGL(2n)) = \bigsqcup_G \overline{\Phi_2}(G).
\end{gather}

Moreover, $\phi$ gives rise to
\begin{align*}
  \phi_O & := \bigoplus_{i \in I_\phi^+} \phi_i \in \Phi_\text{ell,bdd}(\tGL(n_O)),\\
  \phi_S & := \bigoplus_{i \in I_\phi^-} \phi_i \in \Phi_\text{ell,bdd}(\tGL(n_S)).
\end{align*}

\paragraph{The crude correspondence}
Let $G$ be an elliptic endoscopic group as above. Fix the $\theta$-stable $F$-splitting and Whittaker datum as in \S\ref{sec:geom-trans}. Choose compatible Haar measures in the sense of Definition \ref{def:measures}.

\begin{theorem}[Arthur, {\cite[Theorem 1.5.1 and 2.2.1]{ArEndo}}]\label{prop:LLC-SO}
  For each $\phi \in \overline{\Phi_2(G)}$, one can canonically associate a nonempty finite subset $\Pi_\phi$ of $\overline{\Pi_2}(G)$, such that
  \begin{enumerate}
    \item there is the disjoint union
    $$ \overline{\Pi_2}(G) = \bigsqcup_{\phi \in \overline{\Phi_2}(G)} \Pi_\phi ; $$
    \item the linear functional on $\overline{\mathcal{I}}(G)$
      \begin{gather*}
        \Theta_\phi^G: f^G \longmapsto \sum_{\bar{\sigma} \in \Pi_\phi} \Theta^G_{\sigma}(f^G)
      \end{gather*}
      factors through $\overline{S\mathcal{I}}(G)$, i.e.\ $\Theta_\phi^G$ is a stable distribution;
    \item let $f \mapsto f^G$ be the transfer map from $\mathcal{I}(\tGL(n))$ to $\overline{S\mathcal{I}}(G)$ in Theorem \ref{prop:LSK}, if $\pi \in \Pi_{\mathrm{temp}}(\GL(2n))$, $\phi \leftrightarrow \pi$, then
    $$ \Theta^{\tGL(2n)}_{\tilde{\pi}}(f) = \Theta_\phi^G(f^G) $$
    for all $f \in \mathcal{I}(\tGL(2n))$, where $\tilde{\pi}$ is the representation of $\tGL(2n,F)$ defined in Remark \ref{rem:twisted-pi}.
    \item set $G_O := \SO(V,q)$, $G_S := \SO(N_S+1)$, if $f^G$ admits a decomposition $f^G = f^{G_O} \otimes f^{G_S}$ with $f^{G_O} \in \overline{S\mathcal{I}}(G_O)$, $f^{G_S} \in S\mathcal{I}(G_S)$, then
    $$ \Theta_\phi^G(f^G) = \Theta_{\phi_O}^{G_O}(f^{G_O}) \Theta_{\phi_S}^{G_S}(f^{G_S}). $$
  \end{enumerate}
\end{theorem}
The last assertion is trivial for simple endoscopic data. Moreover, $\Pi_\phi$ is characterized by these properties; see \cite[Remark 1 after Theorem 2.2.1]{ArEndo}.

% The shorter remark
\begin{remark}
  To see why this furnishes a local Langlands correspondence, one should regard the $L$-embedding $\xi: \Lgrp{G} \hookrightarrow \GL(2n,\C) \times \text{W}_F$ given by twisted endoscopy. This is explained, for example, in \cite[1.8]{Wa10}. The correspondence is crude in the sense that $\Lgrp{G}$ can admit an extra symmetry in $\Lgrp{\GL(2n)}$ if $\Out(G) \neq \{1\}$. See also \cite[pp.18-19]{KS}.
\end{remark}

\begin{definition}
  Let $\phi \in \Phi_\text{ell,bdd}(\tGL(2n))$. We say that $\phi$ comes from $G$ if $\phi \in \overline{\Phi_2}(G)$. Let $\pi$ be a selfdual representation of $\GL(2n,F)$, we say that $\pi$ comes from $G$ if $\pi \leftrightarrow \phi$ for some $\phi$ coming from $G$.
\end{definition}
By \eqref{eqn:Phi_2-disjoint}, $\phi$ comes from exactly one $G$.

\subsection{Character relations}
The aim of this subsection is to establish a twisted case of \cite[Corollary 6.4]{Ar96} that relates character values. Retain the notations in the preceding subsection and use compatible measures prescribed in Definition \ref{def:measures}.

\paragraph{Measures and integration}
We set up a convenient integration apparatus as in \cite{Ar96}. The sets $\Gamma_\text{reg,ell}(\tGL(2n))$, $\Delta_\text{reg,ell}(\tGL(n))$ acquire quotient topologies from $\tGL(2n)_\text{reg}$. Define Radon measures on these spaces by requiring that
\begin{align}
  \int_{\Gamma_\text{reg,ell}(\tGL(2n))} \alpha(\delta) \dd\delta & = \sum_{\tilde{T}} |W(\GL(2n,F), \tilde{T}(F))|^{-1} \int_{\tilde{T}(F)_{/\theta}} \alpha(t) \dd t, \\
  \int_{\Delta_\text{reg,ell}(\tGL(2n))} \alpha'(\delta') \dd \delta' & = \sum_{\tilde{T}} |W(\GL(2n), \tilde{T})(F)|^{-1} \int_{\tilde{T}_{/\theta}(F)} \alpha'(t) \dd t,
\end{align}
where $\alpha \in C_c(\Gamma_\text{reg,ell}(\tGL(2n)))$, $\alpha' \in C_c(\Delta_\text{reg,ell}(\tGL(2n)))$; the $\tilde{T}$ ranges over conjugacy classes and stable conjugacy classes of elliptic maximal tori in $\tGL(2n)$, respectively. The measures on $T(F)_{/\theta}$, $T_{/\theta}(F)$ and $T_\theta(F)$ are related as in \S\ref{sec:orbint}.

Write $\delta \mapsto \delta'$ if $\delta$ lies in the stable conjugacy class $\delta'$. It follows easily from the definitions above that
\begin{gather}\label{eqn:st-unst-tGL}
  \int_{\Delta_\text{reg,ell}(\tGL(2n))} \left( \sum_{\delta \mapsto \delta'} \alpha(\delta) \right) \dd\delta' = \int_{\Gamma_\text{reg,ell}(\tGL(2n))} \alpha(\delta) \dd\delta, \quad \alpha \in C_c(\Gamma_\text{reg,ell}(\tGL(2n))).
\end{gather}

Using similar notations, define Radon measures on $\Gamma_\text{reg,ell}(G)$, $\Delta_\text{reg,ell}(G)$ such that
\begin{align}\label{eqn:measure-1}
  \int_{\Gamma_\text{reg,ell}(G)} \beta(\gamma) \dd\gamma & = \sum_{T_G} |W(G(F), T_G(F)|^{-1} \int_{T_G(F)} \beta(t) \dd t, \\
  \int_{\Delta_\text{reg,ell}(G)} \beta'(\gamma') \dd \gamma' & = \sum_{T_G} |W(G,T_G)(F)|^{-1} \int_{T_G(F)} \beta'(t) \dd t, \\
  \int_{\Delta_\text{reg,ell}(G)} \left( \sum_{\gamma \mapsto \gamma'} \beta(\gamma) \right) \dd\gamma' &= \int_{\Gamma_\text{reg,ell}(G)} \beta(\gamma) \dd\gamma ,
\end{align}
where $\beta \in C_c(\Gamma_\text{reg,ell}(G))$, $\beta' \in C_c(\Delta_\text{reg,ell}(G))$.

Equip $\overline{\Gamma}_\text{reg,ell}(G)$, $\overline{\Delta}_\text{reg,ell}(G)$ with quotient measures by $\Out_{2n}(G)$; we can also replace $\overline{\Delta}_\text{reg,ell}(G)$ by its open dense subset $\overline{\Delta}_{2n-\text{reg,ell}}(G)$. By taking disjoint union, we obtain a Radon measure on $\Gamma_\text{reg,ell}^{\mathcal{E}}(\tGL(2n))$.

The following result is a twisted analogue of \cite[Lemma 2.3]{Ar96}.
\begin{lemma}\label{prop:change-of-variables}
  Suppose the Haar measures are compatible in the sense of Definition \ref{def:measures}. Let $\alpha \in C_c(\Gamma_{\mathrm{reg,ell}}(\tGL(2n)))$, $\beta \in C_c(\Gamma_{\mathrm{reg,ell}}^{\mathcal{E}}(\tGL(2n)))$, then
  \begin{multline*}
    \int_{\Gamma_{\mathrm{reg,ell}}(\tGL(2n))} \left( \sum_{\gamma \in \Gamma_{\mathrm{reg,ell}}^{\mathcal{E}}(\tGL(2n))} \beta(\gamma) \Delta_\lambda(\gamma,\delta)\alpha(\delta) \right) \dd\delta = \\
    \int_{\Gamma_{\mathrm{reg,ell}}^{\mathcal{E}}(\tGL(2n))} \left( \sum_{\delta \in \Gamma_{\mathrm{reg,ell}}(\tGL(2n))} \beta(\gamma)\Delta_\lambda(\gamma,\delta)\alpha(\delta) \right) \dd\gamma .
  \end{multline*}
\end{lemma}
\begin{proof}
  Using \eqref{eqn:st-unst-tGL}, the left hand side transforms into
  $$ \int_{\Delta_{\mathrm{reg,ell}}(\tGL(2n))} \left( \sum_{\substack{\gamma,\delta \\ \delta \mapsto \delta'}} \beta(\gamma) \Delta_\lambda(\gamma,\delta)\alpha(\delta) \right) \dd\delta' . $$

  Recall that the correspondence $\gamma \leftrightarrow \delta$ is really determined by the stable conjugacy class $\delta'$ such that $\delta \mapsto \delta'$; write this correspondence as $\gamma \leftrightarrow \delta'$. To conclude, we claim that the right hand side is equal to
  $$ \int_{\Delta_{\mathrm{reg,ell}}(\tGL(2n))} \left( \sum_{\gamma \leftrightarrow \delta'} \beta(\gamma) \Delta_\lambda(\gamma,\delta)\alpha(\delta) \right) \dd\delta' . $$

  Indeed, it suffices to concentrate on only one endoscopic group $G$ at a time. We transfer the classes between $\Delta_{2n-\text{reg}}(G)$ and $\Delta_\text{reg}(\tGL(2n))$ using the maps $T_G \rightiso T_{/\theta}$ mentioned in Definition \ref{def:measures} where $T_G \subset G$, $\tilde{T} \subset \GL(2n)$ are appropriate elliptic maximal tori; it factorizes into a map between $F$-varieties $T_G/W(G,T) \to T_{/\theta}/W(\GL(2n),\tilde{T})$. Taking $F$-points induces a well-defined map $\Xi: \Delta_{2n-\text{reg}}(G) \to \Delta_\text{reg}(\tGL(2n))$, which is locally a homeomorphism. We have seen that it factors through the $\Out_{2n}(G)$-action.

  Observe that only those elements $\delta'$ in the image of $\Xi$ contribute to the integral because of the presence of $\Delta_\lambda(\gamma,\delta)$. By Definition \ref{def:measures}, the change of variables by $\Delta_{2n-\text{reg,ell}}(G) \xrightarrow{\Xi} \Delta_\text{reg,ell}(\tGL(2n))$ has jacobian equal to $1$, whence the claim.
\end{proof}

\paragraph{Character values}
Let $\Lambda_G$ be a linear functional $\overline{S\mathcal{I}}(G) \to \C$. There is a canonical way to extend $\Lambda_G$ to $S\mathcal{I}(G)$, namely by setting
\begin{gather}\label{eqn:average}
  \tilde{\Lambda}_G: f \mapsto \Lambda_G \left(\frac{f+s(f)}{2}\right)
\end{gather}
where $s$ is the nontrivial element in $\Out_{2n}(G)$ if $\Out_{2n}(G) \neq \{1\}$; otherwise take $s=\identity$. Thus $\tilde{\Lambda}_G$ can be viewed as an invariant distribution on $G(F)$. We say $\Lambda_G$ is represented by a locally integrable function (resp. locally constant on $G_\text{reg}(F)$) if $\tilde{\Lambda}_G$ is. Note that when $\Lambda_G$ is represented by a locally integrable function, the invariant function is necessarily $\Out_{2n}(G)$-invariant.

In particular, the linear functionals $\Theta^G_\phi$ defined in Theorem \ref{prop:LLC-SO} are represented by locally integrable functions which are locally constant on $G_\text{reg}(F)$.

Consider now a linear functional $\Lambda' = (\Lambda_G)_G$ on $\bigoplus_G \overline{S\mathcal{I}}(G)$. We say $\Lambda'$ is represented by a locally integrable function, etc., if each component $\Lambda_G$ is. If it is indeed the case, we will view $\Lambda'$ as a function on $\bigsqcup_G \overline{\Delta}_\text{reg}(G)$. One can also speak of its restriction to $\Gamma_\text{reg,ell}^{\mathcal{E}}(\tGL(2n))$.

In the same vein, define the Weyl discriminant $D' = (D^G)_G$ as a function on $\bigsqcup_G \overline{\Delta}_\text{reg}(G)$.

\begin{proposition}\label{prop:dist-identity}
  Let $\Lambda$ (resp. $\Lambda' = (\Lambda_G)_G$) be a linear functional on $\mathcal{I}(\tGL(2n))$ (resp. $\bigoplus_G \overline{S\mathcal{I}}(G)$). Assume that $\Lambda'|_{\mathcal{I}_\mathrm{cusp}^{\mathcal{E}}(\tGL(2n))}$ transfers to $\Lambda|_{\mathcal{I}_\mathrm{cusp}(\tGL(2n))}$ in the sense that
  $$ \Lambda'(f') = \Lambda(f), \quad f \in \mathcal{I}_\mathrm{cusp}(\tGL(2n)), \; f \mapsto f'=(f^G)_G; $$
  cf.\ Theorem \ref{prop:transfer-surj}. If $\Lambda$, $\Lambda'$ are both represented by locally integrable functions which are locally constant on regular semisimple elements, then
  $$ |D^{\tGL(2n)}(\delta)|^{\frac{1}{2}}\Lambda(\delta) = \sum_{\gamma \in \Gamma_{\mathrm{reg,ell}}^{\mathcal{E}}(\tGL(2n))} |D'(\gamma)|^{\frac{1}{2}} \Lambda'(\gamma) \Delta_\lambda(\gamma,\delta), \quad \delta \in \Gamma_{\mathrm{reg,ell}}(\tGL(2n)). $$
\end{proposition}
\begin{proof}
  In what follows, we treat all $G$ simultaneously by using the language in Theorem \ref{prop:transfer-surj} and \eqref{eqn:LSK-char}.

  Take $f \in \mathcal{I}_\text{cusp}(\tGL(2n))$. By Proposition \ref{prop:Weyl-int} and the definition of measures on $\Gamma_\text{reg,ell}(\tGL(2n))$, we have
  $$ \Lambda(f) = \int_{\Gamma_\text{reg,ell}(\tGL(2n))} |D^{\tGL(2n)}(\delta)|^{\frac{1}{2}}\Lambda(\delta) I^{\tGL(2n)}(\delta,f) \dd\delta. $$

  Similarly, the untwisted case of Proposition \ref{prop:Weyl-int-st}, applied simultaneously to each $G$, implies
  $$ \Lambda'(f') = \int_{\Gamma_\text{reg,ell}^{\mathcal{E}}(\tGL(2n))} |D'(\gamma)|^{\frac{1}{2}}\Lambda'(\gamma) S'(\gamma, f') \dd\gamma . $$
  Applying \eqref{eqn:LSK-char}, $\Lambda'(f')$ is equal to
  $$ \int_{\Gamma_\text{reg,ell}^{\mathcal{E}}(\tGL(2n))} |D'(\gamma)|^{\frac{1}{2}} \Lambda'(\gamma) \sum_{\delta \in \Gamma_\text{reg,ell}(\tGL(2n))} \Delta_\lambda(\gamma,\delta) I^{\tGL(2n)}(\delta,f) \dd\delta .$$

  Now we can apply Lemma \ref{prop:change-of-variables} to get
  $$ \Lambda'(f') = \int_{\Gamma_\text{reg,ell}(\tGL(2n))} \left( \sum_{\gamma \in \Gamma_\text{reg,ell}^{\mathcal{E}}(\tGL(2n))} |D'(\gamma)|^{\frac{1}{2}}\Lambda'(\gamma) \Delta_\lambda(\gamma,\delta) \right) I^{\tGL(2n)}(\delta,f) \dd\gamma . $$

  Let $\delta_0 \in \Gamma_\text{reg,ell}(\tGL(2n))$. Replace $f$ by a sequence $\{f_i\}_{i=1}^\infty$ in $\mathcal{I}_\text{cusp}(\tGL(2n))$ such that $I^{\tGL(2n)}(\cdot,f_i)$ approaches the Dirac measure concentrated at $\delta_0$ as $i \to \infty$, then the equality $\Lambda(f_i) = \Lambda'(f'_i)$ and the equations above show
  $$ |D^{\tGL(2n)}(\delta_0)|^{\frac{1}{2}}\Lambda(\delta_0) = \sum_{\gamma \in \Gamma_\text{reg,ell}^{\mathcal{E}}(\tGL(2n))} |D'(\gamma)|^{\frac{1}{2}} \Lambda'(\gamma) \Delta_\lambda(\gamma,\delta_0), $$
  as asserted.
\end{proof}

For $\phi \in \overline{\Phi_2}(G)$, define
\begin{gather*}
  S^G(\phi,\gamma) := |D^G(\gamma)|^{\frac{1}{2}}\Theta^G_\phi(\gamma), \quad \gamma \in \overline{\Delta}_{\mathrm{reg,ell}}(G).
\end{gather*}

\begin{corollary}\label{prop:char-identity}
  Let $\pi \in \Pi_\mathrm{temp}(\GL(2n))$ be selfdual with $L$-parameter $\phi \in \Phi_\mathrm{ell,bdd}(\tGL(2n))$. Let $G$ be the elliptic endoscopic group such that $\phi \in \overline{\Phi_2}(G)$, then
  $$ I^{\tGL(2n)}(\tilde{\pi},\delta) = \sum_{\gamma \in \overline{\Delta}_{\mathrm{reg,ell}}(G)} S^G(\phi,\gamma) \Delta_\lambda(\gamma,\delta) , \quad \delta \in \Gamma_{\mathrm{reg,ell}}(\tGL(2n)). $$
  Here $\tilde{\pi}$ is the representation of $\tGL(2n,F)$ defined in Remark \ref{rem:twisted-pi}.
\end{corollary}
\begin{proof}
  Take $\Lambda' := \Theta_\phi^G$, regarded as a linear functional on $\bigoplus_{G_1} \overline{S\mathcal{I}}(G_1)$ concentrated at the $G$-slot. By Theorem \ref{prop:LLC-SO}, $\Lambda'(f') = \Lambda(f)$ for all $f \in \mathcal{I}(\tGL(2n))$. Now Proposition \ref{prop:dist-identity} can be applied.
\end{proof}
These character identities are independent of choice of Haar measures.

\section{The Goldberg-Shahidi-Spallone formalism}\label{sec:GSS-formalism}
The setting below is modeled upon \cite{Sp11}, corresponding to the special case ``$n=2m$'' (resp. ``$n=2m+1$'') in the terminology of \cite{GS98,GS01} (resp. of \cite{GS09}).

\subsection{Sesquilinear algebra}\label{sec:grouptheory}
Let $E/F$ be a field extension with $[E:F] \leq 2$. As usual, set $\tau$ to be the nontrivial element in $\text{Gal}(E/F)$ if $[E:F]=2$, otherwise $\tau := \identity$. Fix $\epsilon = \pm 1$ and consider a $(E,\tau)$-hermitian space $(V_1,q_1)$ of sign $\epsilon$. There are four cases.
\begin{enumerate}
  \item Even orthogonal case: $E=F$, $\epsilon=1$, $\dim_F V_1 = 2k$ for some $k$.
  \item Odd orthogonal case: $E=F$, $\epsilon=1$, $\dim_F V_1 = 2k+1$ for some $k$.
  \item Symplectic case: $E=F$, $\epsilon=-1$.
  \item Hermitian/anti-hermitian case: $[E:F]=2$.
\end{enumerate}

Assume that there exists a decomposition of $E$-vector spaces
$$ V_1 = H' \oplus V \oplus H $$
such that
\begin{itemize}
  \item $\dim_E H = \dim_E H' = \dim_E V$;
  \item $H,H'$ are totally isotropic in $V_1$;
  \item set $q := q_1|_V$, then $(V,q)$ is a $E/F$-hermitian space of sign $\epsilon$;
  \item $V$ is orthogonal to $H \oplus H'$;
  \item in the even orthogonal case, $(V,q)$ is not isotropic of dimension two.
\end{itemize}
This leads to an identification $H' = H^\vee$, namely an element $v' \in H'$ corresponds to the $E$-linear functional
$$ v \longmapsto q(v'|v), \quad v \in H . $$
Henceforth we abandon the notation $H'$ and write $H^\vee$ instead. Note that for $Y \in \Hom_E(H,H^\vee) \hookrightarrow \End_E(V_1)$, the adjoint map of $Y$ with respect to $q_1$ is then identified with $\epsilon \check{Y}: H \to H^\vee$.

Set $G_1 := \U(V_1,q_1)^\circ$. Given the decomposition above for $V_1$, we define the following subgroups of $G_1$:
\begin{align*}
  P & := \Stab_{G_1} (H^\vee), \\
  P^- & := \Stab_{G_1} (H), \\
  M & := P \cap P^- = \GL_E(H) \times \U(V,q)^\circ.
\end{align*}

The groups $P,P^-$ are maximal parabolic subgroups of $G_1$ with the common Levi component $M$. Write $P=MU$, $P^- = MU^-$ for the corresponding Levi decompositions.

Write $\sigma[q]$ for the element in $\Isom_E(V,V^\vee)$ corresponding to $q$, that is, $\angles{\sigma[q]v,v'}=q(v|v')$ for all $v,v'\in V$. 

\begin{proposition}\label{prop:XY-cond}
  The elements of $U$ are in natural bijection with pairs $(X,Y)$ such that
  \begin{gather*}
    X \in \Hom_E(V,H^\vee), \\
    Y \in \Hom_E(H,H^\vee),
  \end{gather*}
  and
  \begin{gather}\label{eqn:XY-cond}
    Y + \epsilon\check{Y} + X \sigma[q]^{-1} \check{X} = 0 .
  \end{gather}
  More precisely, given $(X,Y)$ as above, set $X' := -\sigma[q]^{-1}\check{X} \in \Hom_E(H,V)$. Then $(X,Y,X')$ defines an element $n(X,Y) \in \End_E(V_1)$ such that $n(X,Y)|_H = Y + X'$, $n(X,Y)|_V = X$, $n(X,Y)|_{H^\vee}=0$. The corresponding element is $u(X,Y) := 1+n(X,Y)$.
\end{proposition}
\begin{proof}
  This is \cite[\S 2.1]{Sp11}.
\end{proof}

%There is a simpler description of $\mathfrak{u} := \Lie U$.
%\begin{proposition}\label{prop:AB-cond}
%  The elements in $\mathfrak{u}(F)$ are in natural bijection with pairs $(A,B) \in \Hom_E(V,H^\vee) \times \Hom_E(H,H^\vee)$ such that
%  \begin{gather}\label{eqn:AB-cond}
%    B + \epsilon\check{B}=0.
%  \end{gather}
%  Given $(A,B)$ as above, we construct the triplet $(A,B,A')$ and the element $n(A,B) \in \End_E(V_1)$ in precisely the same manner as in Proposition \ref{prop:XY-cond}, then $n(A,B)$ is the corresponding element in $\mathfrak{u}$.
%\end{proposition}
%\begin{proof}
%  This is \cite[\S 2.2]{Sp11}.
%\end{proof}
%To describe elements in $U^-$ and $\mathfrak{u}^-$, simply exchange the roles of $H,H^\vee$. Observe that our constructions are functorial in $F$, thus determines isomorphisms between $F$-varieties.
Introduce the twisted space $(\GL_E(H),\tGL_E(H))$ of non-degenerate $E/F$-sesquilinear forms on $H$ (recall \S\ref{sec:twisted-eg}).

\begin{lemma}\label{prop:rigidify}
  Let $U'$ be the Zariski open dense subset of $U$ consisting of elements of the form $u(X,Y)$, in the notation of Proposition \ref{prop:XY-cond}, such that $X,Y$ are both invertible. There is a canonical isomorphism between $F$-varieties
  \begin{align*}
    U' & \stackrel{\sim}{\longrightarrow} \left\{ (\delta, \varphi) :
      \begin{array}{l}
        \delta: \in \tGL_E(H), \\
        \varphi \in \Isom_{E,\tau}((H, \delta + \epsilon\cdot {}^t \delta), (V,-\epsilon q)).
      \end{array}
    \right\} \\
    u(X,Y) & \longmapsto \begin{cases}
      \delta :  (v,v') \mapsto \angles{Yv,v'} \\
      \varphi  := \sigma[q]^{-1} \check{X}.
     \end{cases}
  \end{align*}
\end{lemma}
\begin{proof}
  By recalling our identification $H = H^{\vee\vee}$ and the definition of ${}^t \delta$, one sees that ${}^t \delta (v|v') = \tau\angles{Yv',v} = \angles{\check{Y}v, v'}$. By the same reasoning, one can rewrite \eqref{eqn:XY-cond} as
  \begin{align*}
    \angles{(Y+\epsilon\hat{Y})v, v'} & = -\angles{X \sigma[q]^{-1} \check{X} v, v'} \\
    & = -\tau\angles{\check{X}v', \sigma[q]^{-1} \check{X}v} \\
    & = -\tau\angles{\sigma[q] \varphi v', \varphi v} \\
    & = -\tau q(\varphi v'|\varphi v) = -\epsilon q(\varphi v|\varphi v')
  \end{align*}
  for all $v,v' \in H$. The invertibility of $Y+\epsilon\check{Y}$ follows automatically. This is enough to complete the proof.
\end{proof}

\subsection{GS-norms}
\begin{definition}
  Let $(X,Y)$ be as in Proposition \ref{prop:XY-cond} such that $X,Y$ are both invertible. Define the Goldberg-Shahidi norm (abbreviated as GS-norm) by
  \begin{gather}\label{eqn:GS-norm}
    \text{Norm}(X,Y) := 1 + \sigma[q]^{-1} \check{X} Y^{-1} X \in \End_E(V).
  \end{gather}
  We also regard $\text{Norm}$ as a morphism $U' \to \End_E(V)$.
\end{definition}

We will identify $\Isom_E(H,H^\vee)$ and $\tGL_E(H)$ by sending $Y: H \rightiso H^\vee$ to the sesquilinear form $\delta_Y:(v,v') \mapsto \angles{Yv,v'}$.

\begin{theorem}[{\cite[\S 3, \S 5]{Sp11}}]\label{prop:Norm}
  The morphism $\mathrm{Norm}$ has the following properties.
  \begin{enumerate}
    \item The image of $\mathrm{Norm}$ is $\{\gamma \in \U(V,q):  \det(\gamma-1) \neq 0 \}$; this subset is contained in ${-1 \cdot \U(V,q)^\circ}$ if $(V,q)$ is odd orthogonal, otherwise it is contained in $\U(V,q)^\circ$.
    \item Let $(X,Y)$ be as in Proposition \ref{prop:XY-cond}, $g \in \GL_E(V)$, then $(gX, gY\check{g})$ also satisfies \eqref{eqn:XY-cond}, and
    $$ \mathrm{Norm}(gX, gY\check{g}) = \mathrm{Norm}(X,Y). $$
    \item Let $X \in \Isom_E(V,H^\vee)$, then the morphism
    \begin{gather}\label{eqn:section}
      \gamma \longmapsto \delta_X(\gamma) := X(\gamma-1)^{-1}\sigma[q]^{-1}\check{X} \in \Isom_E(H,H^\vee)
    \end{gather}
    is a section of $\mathrm{Norm}$.
    \item $\mathrm{Norm}$ induces a surjection
    $$\xymatrix{
      \GL_E(H) \backslash \{ \delta \in \tGL_E(H) : \delta=\delta_Y \text{ for some } (X,Y) \text{ from } U'  \} \ar@{->>}[d] & \ni & (X,Y) \ar@{|->}[d] \\
      \U(V,q) \backslash \{ \gamma \in \U(V,q):  \det(\gamma-1) \neq 0 \} & \ni & \mathrm{Norm}(X,Y)
    }$$
    where $\GL_E(V)$ (resp. $\U(V,q)$) acts by conjugation.
  \end{enumerate}
\end{theorem}
Let $\delta \in \tGL_E(H)$, $\gamma \in \U(V,q)$, we will write
$$ \delta \GSnorm \gamma $$
if their conjugacy classes correspond as in the last assertion.
\begin{proof}
  These assertions are essentially due to \cite{GS98,GS01,GS09} and proved in \cite{Sp11} using a coordinate-free approach. Note that in the last assertion, the $\GL_E(H)$-action is well-defined by the second assertion and \eqref{eqn:Isom-bitorsor}.
\end{proof}

The crux is to compare the GS-norm with the correspondence of conjugacy classes in twisted endoscopy. This is done in \cite{GS98,GS01,GS09} using the language of \cite{KS}, in a broader context. Our concern here is to describe $\text{Norm}$ directly in terms of twisted spaces and the parametrization in \S\ref{sec:classes}. This will also give a transparent explanation for Theorem \ref{prop:Norm}, at least for very regular semisimple classes.

First of all, as the map in the last assertion of Theorem \ref{prop:Norm} is obtained from a morphism between $F$-varieties, we actually get a map between stable conjugacy classes. Let $\delta \in \tGL_E(H)$ be very regular semisimple. We invoke the parametrization in \S\ref{sec:classes}.
\begin{enumerate}
  \item In the odd orthogonal case, the stable conjugacy class of $\delta$ is parametrized by a quadruplet $(L,L_\pm,x,x_D)$;
  \item In the remaining cases, the stable conjugacy class of $\delta$ is parametrized by a triplet $(L,L_\pm,x)$.
\end{enumerate}
%Recall the basic properties of the parameters:
%\begin{itemize}
%  \item $L$, $L_\pm$ are étale $F$-algebras;
%  \item $\dim_E L = \dim_E H - 1$ in the odd orthogonal case, otherwise $\dim_E L = \dim_E H$;
%  \item $L_\pm$ is the fixed subalgebra by an involution $\tau$ extending $\tau \in \text{Gal}(E/F)$;
%  \item $L=F(x)$ and $y\tau(x)=1$.
%  \item $x_D \in F^\times/F^{\times 2}$ in the odd orthogonal case.
%\end{itemize}

\begin{lemma}\label{prop:GS-correspondence}
  Let $\delta \in \tGL_E(H)$ be very regular semisimple, such that $\delta=\delta_Y$ for some $(X,Y)$ satisfying \eqref{eqn:XY-cond}. Assume that the stable conjugacy class of $\delta$ is parametrized by $(L,L_\pm,x)$ (resp. $(L,L_\pm,x,x_D)$ in the odd orthogonal case).
  \begin{enumerate}
    \item In the odd orthogonal case, suppose that $(-1) \cdot \mathrm{Norm}(X,Y)$ is very regular semisimple in $\SO(V,q)$, then its stable conjugacy class is parametrized by $(L,L_\pm,y)$ where
      $$ y = \frac{\tau(x)}{x}; $$
    \item In the remaining cases, suppose that $\mathrm{Norm}(X,Y)$ is very regular semisimple in $\U(V,q)^\circ$, then its stable conjugacy class (up to $\Or(V,q)$ in the even orthogonal case) is parametrized by $(L,L_\pm,y)$ where
      $$ y = -\frac{\epsilon\tau(x)}{x}. $$
  \end{enumerate}
\end{lemma}
Since very regular semisimple elements form Zariski open dense subsets, the condition in this lemma holds generically.
\begin{proof}
  We start from the second case. Implicit in the parametrization of $\delta$ is an identification $H = L$ as $E$-vector spaces. We deduce an isomorphism $\varphi_L: H^\vee = L^\vee \rightiso L$, characterized by
  \begin{gather}\label{eqn:varphi_L}
    \angles{\check{v},v} = \Tr_{L/E}(\tau(\varphi_L \check{v})v), \quad \check{v} \in L^\vee, v \in L .
  \end{gather}

  Since $\delta$ is very regular, we have $\delta + \epsilon \cdot {}^t \delta \in \tGL_E(H)$. Moreover, Remark \ref{rem:transposed-form} affirms that $\delta + \epsilon \cdot {}^t \delta \in \tGL_E(H)$ is parametrized by $(L,L_\pm,x+\epsilon\tau(x))$. This element corresponds to $Y + \epsilon\check{Y} \in \Isom_E(H,H^\vee)$. Hence we have the following commutative diagram with invertible arrows.
  $$\xymatrix@R=5em@C=5em{
    V \ar[r]^{\check{X}Y^{-1}X} \ar[d]_{X} & V^\vee \ar[r]^{\sigma[q]^{-1}} & V \ar[d]^X \\
    H^\vee \ar[d]_{\varphi_L} \ar@<0.7ex>[r]^{Y^{-1}} & H \ar[u]^{\check{X}} \ar[r]^{-(Y+\epsilon \check{Y})} \ar@<0.7ex>[l]^{Y} & H^\vee \ar[d]^{\varphi_L} \\
    L \ar@<0.7ex>[r]^{x^{-1} \cdot} & L \ar@{=}[u] \ar[r]^{-(x+\epsilon\tau(x)) \cdot} \ar@<0.7ex>[l]^{x \cdot} & L
  }$$
  where we have used \eqref{eqn:XY-cond}. In view of the definition of GS-norms in \eqref{eqn:GS-norm}, we get the commutative diagram
  $$\xymatrix@R=3em@C=7em{
    V \ar[r]^{\text{Norm}(X,Y)} \ar[d]_{\varphi_L \circ X} & V \ar[d]^{\varphi_L \circ X} \\
    L \ar[r]_{(1 - x^{-1}(x+\epsilon\tau(x))) \cdot } & L
  }.$$
  We have $1-x^{-1}(x+\epsilon\tau(x)) = -\epsilon x^{-1}\tau(x) = y$. The parametrization in \S\ref{sec:classes} says that $(L,L_\pm,y)$ parametrizes $\text{Norm}(X,Y)$.

  Consider the odd orthogonal case now. We have $E=F$, $\epsilon=1$ and an identification $H=L \oplus F$ as $F$-vector spaces. In this case, we deduce an isomorphism of $F$-vector spaces
  $$ \varphi_L = (\varphi_L^1, \varphi_L^2): H^\vee = L^\vee \oplus F^\vee \rightiso L \oplus F,$$
  where $\varphi_L^1: L^\vee \rightiso L$ is defined by \eqref{eqn:varphi_L} and $\varphi_L^2: F^\vee \rightiso F$, or rather its inverse, corresponds to the quadratic form $x \mapsto x^2$ on $F$.

  As before, we obtain a commutative diagram
  $$\xymatrix@R=5em@C=7em{
    V \ar[r]^{\check{X}Y^{-1}X} \ar[d]_{X} & V^\vee \ar[r]^{\sigma[q]^{-1}} & V \ar[d]^X \\
    H^\vee \ar[d]_{\varphi_L} \ar@<0.7ex>[r]^{Y^{-1}} & H \ar[u]^{\check{X}} \ar[r]^{-(Y+\check{Y})} \ar@<0.7ex>[l]^{Y} & H^\vee \ar[d]^{\varphi_L} \\
    L \oplus F \ar@<0.7ex>[r]^{(x^{-1},x_D^{-1}) \cdot} & L \oplus F \ar@{=}[u] \ar[r]^{(-(x+\tau(x)), -2x_D) \cdot} \ar@<0.7ex>[l]^{(x,x_D) \cdot} & L \oplus F
  }.$$
  Hence
  $$\xymatrix@R=3em@C=7em{
    V \ar[r]^{\text{Norm}(X,Y)} \ar[d]_{\varphi_L \circ X} & V \ar[d]^{\varphi_L \circ X} \\
    L \oplus F \ar[r]_{(-x^{-1}\tau(x),-1) \cdot } & L \oplus F
  }.$$
  To conclude, it suffices to note that $(-x^{-1}\tau(x),-1)=-(y,1)$.
\end{proof}

\begin{corollary}\label{prop:GS-ellipticity}
  Suppose $\delta \in \tGL_E(H)$, $\gamma \in \U(V,q)^\circ$ (resp. $-\gamma \in \U(V,q)^\circ$ in the odd orthogonal case) are both very regular semisimple. If $\delta \GSnorm \gamma$, then $\tGL_E(H)_\delta \simeq \U(V,q)_\gamma$. In particular, $\delta$ is elliptic if and only if $\gamma$ (resp. $-\gamma$ in the odd orthogonal case) is.
\end{corollary}
\begin{proof}
  The first assertion follows from the Proposition \ref{prop:centralizer-desc} and Lemma \ref{prop:GS-correspondence}. The second assertion follows from the easy fact that the identity components of $Z_{\U(V,q)^\circ}$ and $Z_{\tGL_E(V)}$ are both anisotropic.
\end{proof}

\begin{corollary}\label{prop:GS-KS}
  In the even orthogonal case, suppose $\delta \in \tGL(H)$, $\gamma \in \SO(V,q)$ are very regular semisimple. If $\delta \GSnorm \gamma$, then $\delta$ corresponds to $\gamma^{-1}$ (or equivalently $\gamma$) in the sense of twisted endoscopy for $\tGL(H)$.
\end{corollary}
\begin{proof}
  Compare Lemma \ref{prop:GS-correspondence} and Definition \ref{def:correspondence}. Note that $\gamma$ and $\gamma^{-1}$ are stably conjugate up to $\Or(V,q)$. Indeed, their parameters $(L,L_\pm,y)$, $(L,L_\pm,\tau(y))$ are equivalent via $\tau$.
\end{proof}
%\begin{remark}
%  The other cases can also be related to twisted endoscopy of $\tGL_E(H)$ by using the description in \cite[\S 1.9]{Wa10}.
%\end{remark}

\subsection{Transfer factor in the even orthogonal case}\label{sec:Delta-formula2}
We rejoin the setting of \S\ref{sec:Delta-formula1} in this subsection, namely we consider an $F$-quadratic space $(V,q)$ such that $\dim_F V=2n=\dim_F H$ and $\SO(V,q)$ is quasisplit. Regard $\SO(V,q)$ as an endoscopic group of $\tGL(H)$ for some $(2n,0,\chi) \in \mathcal{E}_\text{sim}(2n)$. Choose a basis for $H$ and fix the $\theta$-stable Whittaker datum $(B,\lambda)$ of $\GL(H)$ as in \S\ref{sec:Delta-formula1}, by which the Whittaker-normalized transfer factor $\Delta_\lambda$ is defined.

As usual, we set $K=F \times F$ if $\chi=1$, otherwise $K$ is defined to be the quadratic extension attached to $\chi$. In either case, $x \mapsto N_{K/F}(x)$ denotes the norm form on $K$. Therefore we can write
\begin{gather}\label{eqn:V-decomposition}
  (V,q) \simeq (n-1)\Hy \oplus cN_{K/F}
\end{gather}
for some $c \in F^\times/F^{\times 2}$.

\begin{proposition}\label{prop:Delta-formula-GS}
  Let $\delta \in \tGL(H)$, $\gamma \in \SO(V,q)$ be both very regular semisimple. If $\delta \GSnorm \gamma$ as in Theorem \ref{prop:Norm}, then
  $$ \Delta_\lambda(\gamma^{-1},\delta) = \gamma_{\psi_F}(2(-1)^n q). $$
\end{proposition}
\begin{proof}
  Set $q_\delta := \frac{1}{2}(\delta + {}^t \delta)$ as in Theorem \ref{prop:wald-evenso}. The assumption on $(\gamma,\delta)$, Lemma \ref{prop:rigidify} together with Proposition \ref{prop:Delta-formula-gen} imply
    $$\Delta(\gamma^{-1},\delta) = \begin{cases}
    1, & \text{if } (V, -\frac{1}{2}q) \stackrel{\mathrm{Witt}}{\sim} (-1)^n N_{K/F}, \\
    -1, & \text{otherwise}.
  \end{cases}$$

  However $(V,-\frac{1}{2}q) \stackrel{\mathrm{Witt}}{\sim} (-1)^n N_{K/F}$ if and only if $(V, 2(-1)^{n-1} q) \stackrel{\mathrm{Witt}}{\sim} N_{K/F}$. Put
  $$ (V',q') := (n-1)\Hy \oplus N_{K/F}. $$
  It remains to decide whether $(V, 2(-1)^{n-1} q) \simeq (V',q')$ or not. According to \eqref{eqn:V-decomposition}, $d_\pm(2(-1)^{n-1} q)=d_\pm(q')$, thus it suffices to compare their Hasse invariants $s(2(-1)^{n-1} q)$ and $s(q')$. To this end we use \cite[Proposition 1.3.4]{Pe81}, which implies
  $$ \Delta(\gamma^{-1},\delta) = \frac{s(q')}{s(2(-1)^{n-1}q)} = \frac{\gamma_{\psi_F}(N_{K/F})}{\gamma_{\psi_F}(2(-1)^{n-1}q)} = \gamma_{\psi_F}(N_{K/F}) \gamma_{\psi_F}(2(-1)^n q) , $$
  where we have used the fact that $\gamma_{\psi_F}(\cdot)$ factors through the Witt group.

  By Definition \ref{def:Whittaker-normalization} and Remark \ref{rem:epsilon-Weil}, we arrive at
  \begin{align*}
    \Delta_\lambda(\gamma^{-1},\delta) & = \varepsilon\left(\frac{1}{2}, \chi, \psi_F\right)^{-1} \Delta(\gamma^{-1},\delta) \\
    & = \gamma_{\psi_F}(N_{K/F})^{-1} \Delta(\gamma^{-1},\delta) = \gamma_{\psi_F}(2(-1)^n q),
  \end{align*}
  as required.
\end{proof}

\subsection{Goldberg-Shahidi pairing}\label{sec:pairing}
%\section{The Goldberg-Shahidi pairing}\label{sec:pairing}
%\subsection{The set-up}
%To begin with, we retain the formalism in \S\ref{sec:GSS-formalism} and give a rapid review of the residue of intertwining operators. Consider:
We are now ready to introduce the integral pairing of Goldberg and Shahidi. Consider the general case
\begin{itemize}
  \item $G_1$, $P$, $M=\GL_E(H) \times \U(V,q)^\circ$: as in \S\ref{sec:grouptheory};
  \item $w_0$: the nontrivial element in $W^{G_1}(M) := N_{G_1(F)}(M(F))/M(F)$;
  \item $(\pi \extotimes \sigma, V_\pi \otimes V_\sigma)$: irreducible supercuspidal representation of $M = \GL_E(H) \times \U(V,q)^\circ$ satisfying $w_0 \cdot (\pi \extotimes \sigma) \simeq \pi \extotimes \sigma$;
  \item $(B,\lambda)$: a $\theta$-stable Whittaker datum for $\tGL_E(H)$ (recall Definition \ref{def:Whittaker-normalization}, replace $F$ by $E$ if necessary).
\end{itemize}

The condition $w_0 \cdot (\pi \extotimes \sigma) \simeq \pi \extotimes \sigma$ implies $\pi$ is unitary, hence tempered. Therefore $\pi$ can be canonically extended to a representation $\tilde{\pi}$ of $\tGL_E(H)$ by Remark \ref{rem:twisted-pi}.

Let $H_M: M(F) \to \mathfrak{a}_{M,\C}$ be the Harish-Chandra map, $\lambda \in \mathfrak{a}_{M,\C}^*$, and set $(\pi \extotimes \sigma)_\lambda := (\pi \extotimes \sigma) \otimes e^{\angles{\lambda, H_M(\cdot)}}$. Denote the normalized parabolic induction of $(\pi \extotimes \sigma)_\lambda$ by $\mathcal{I}_P^{G_1}((\pi \extotimes \sigma)_\lambda)$. Choose a representative $\hat{w}_0$ of $w_0$ in $G_1(F)$. We want to study the residue at $\lambda=0$ of the standard intertwining operator
\begin{align*}
  J_P(w_0, (\pi \extotimes \sigma)_\lambda): \mathcal{I}_P^{G_1}((\pi \extotimes \sigma)_\lambda) & \longrightarrow \mathcal{I}_P^{G_1}((\pi \extotimes \sigma)_\lambda) \\
  f & \longmapsto \left[ x \mapsto \int_{U(F)} f(\hat{w}_0^{-1} u x) \dd u \right].
\end{align*}
This integral is absolutely convergent if $\text{Re}\lambda$ lies in the positive chamber, for general $\lambda$ it is obtained by meromorphic continuation.

In the study of $\text{Res}_{\lambda=0} J_P(w_0, (\pi \extotimes \sigma)_\lambda)$, Goldberg and Shahidi were led to study an intricate pairing $R: \mathcal{A}(\tilde{\pi}) \times \mathcal{A}(\sigma) \to \C$ between matrix coefficients (recall the definition of twisted matrix coefficients in \S\ref{sec:twisted-rep}) using GS-norms, which has a regular (or elliptic) part $R_\text{ell}$ that is defined as an integral pairing between orbital integrals. As mentioned in \S\ref{sec:intro}, the non-vanishing of $R_\text{ell}$ is conjectured to be related to twisted endoscopic transfer \cite{Sh95,GS98,GS01,GS09}. In view of Spallone's improved formulae \cite[Corollary 6]{Sp11}, $R_\text{ell}$ is proportional to another pairing $R^\text{ell}$, which we set off to define.

Henceforth, we assume $E=F$ in order to apply Spallone's formulae in \cite[\S 9]{Sp11}.

Fix Haar measures on the groups $\GL(H)$, $\U(V,q)^\circ$, $Z_{U(V,q)^\circ}(F)$ and $Z_{\GL(H)}(F) = F^\times$.

\begin{definition}
  Assume $E=F$. Let $f_{\tilde{\pi}} \in \mathcal{A}(\tilde{\pi})$, $f_\sigma \in \mathcal{A}(\sigma)$. Define
  \begin{multline}\label{eqn:Rell}
    R^\text{ell}(f_{\tilde{\pi}},f_\sigma) = \sum_{\substack{T \\ \text{elliptic}}} |W(U(V,q)^\circ, T(F))|^{-1} \int_{T(F)} \\
      \left( |D^{\U(V,q)^\circ}(\gamma)|^{\frac{1}{2}} \int_{Z_{\U(V,q)^\circ}(F) \backslash \U(V,q)} f_\sigma(y^{-1}\gamma y) \dd y \right) \left( |D^{\tGL(H)}(\delta)|^{\frac{1}{2}} \int_{F^\times \backslash \GL(H)} f_{\tilde{\pi}}(x^{-1}\delta x) \dd x \right)  \dd\gamma,
  \end{multline}
  where
  \begin{itemize}
    \item $T$ ranges over conjugacy classes of elliptic maximal tori in $\U(V,q)^\circ$, whose measures are normalized so that $\text{vol}(T(F)/Z_{\U(V,q)^\circ}(F))=1$;
    \item $\delta \in \tGL(H)$ is given by choosing $X \in \Isom_F(V,H^\vee)$ and setting
      $$\delta := \begin{cases}
        \delta_X(-\gamma), & \text{ in the odd orthogonal case,} \\
        \delta_X(\gamma), & \text{ otherwise };
      \end{cases}$$
      here $\delta_X(\cdot)$ is the section of GS-norm defined in Theorem \ref{prop:Norm}.
  \end{itemize}
\end{definition}

This is the main object is this article. To justify the definition, let us show that $R^\text{ell}$ is indeed the elliptic part of the pairing in \cite{Sp11}, up to a harmless constant. Write $\dim_F V=2n$ (resp. $2n+1$) in the even orthogonal or symplectic case (resp. the odd orthogonal case). Choose $f^{\tilde{\pi}} \in C_c^\infty(\tGL(H))$ such that
$$ f_{\tilde{\pi}}(\delta) = \int_{F^\times} \omega_\pi(z)^{-1} f^{\tilde{\pi}}(z\delta) \dd z, \quad \delta \in \tGL(H)_\text{reg}. $$

Spallone defined a pairing $R(f^{\tilde{\pi}}, f_\sigma)$ in \cite[Corollary 6]{Sp11}. Translated into the language of twisted spaces using \S\ref{sec:nonconnected}, the elliptic part of $R(f^{\tilde{\pi}}, f_\sigma)$ is
\begin{multline*}
  R_\text{ell}(f^{\tilde{\pi}}, f_\sigma) := \Res_{s=0} \sum_{k=0}^\infty q_F^{-2nks} \sum_{\substack{T \\ \text{elliptic}}} |W(U(V,q)^\circ, T(F))|^{-1} \int_{T(F)}  |D^{\U(V,q)^\circ}(\gamma)|^{\frac{1}{2}} |D^{\tGL(H)}(\delta)|^{\frac{1}{2}} \\ \iint_{\substack{x \in T(F) \backslash \GL(H) \\ y \in T(F) \backslash \U(V,q)}}
  \sum_{\alpha \in F^\times/F^{\times 2}} \omega_\pi(\alpha)^{-1} f^{\tilde{\pi}}(\alpha x^{-1}\delta x) f_\sigma(y^{-1}\gamma y) w_k(x,y) \dd x \dd y \dd \gamma
\end{multline*}
where $w_k(x,y)$ is Spallone's weight factor. Here $T$ is identified with $\GL(H)_\delta$ via Corollary \ref{prop:GS-ellipticity}. We only need two properties of $w_k(x,y)$:
\begin{itemize}
  \item $0 \leq w_k(x,y) \leq \text{vol}(T(F))$;
  \item $w_k(x,y)$ converges to $\text{vol}(T(F))$ pointwise, for elliptic $T$.
\end{itemize}
For each $T$, the corresponding sums inside $\Res_{s=0}$ are absolutely convergent when $\text{Re}(s)>0$, by \cite[Proposition 14]{Sp08}.

\begin{proposition}
  We have
  $$ R_\mathrm{ell}(f^{\tilde{\pi}}, f_\sigma) = \begin{cases}
    (|2|_F 2n \log q_F)^{-1} R^\mathrm{ell}(f_{\tilde{\pi}},f_\sigma), & \text{ in the even orthogonal or symplectic case}; \\
    2(|2|_F 2n \log q_F)^{-1} R^\mathrm{ell}(f_{\tilde{\pi}},f_\sigma), & \text{ in the odd orthogonal case}.
  \end{cases}$$
\end{proposition}
\begin{proof}
  The argument is almost identical to that in \cite[\S 3]{SS10}. Fix an elliptic torus $T$ in $\U(V,q)^\circ$ and write its contribution in $R_\text{ell}$ as $\Res_{s=0} \sum_{k=0}^\infty q_F^{-2nks} a_k^T $. Then $a_k^T$ is bounded by a linear combination of terms
  $$ \int_{T(F)} |D^{\U(V,q)^\circ}(\gamma)|^{\frac{1}{2}} |D^{\tGL(H)}(\delta)|^{\frac{1}{2}}  O^{\tGL(H)}_\delta(|f^{\tilde{\pi}}(\alpha\cdot\;)|) O^{\U(V,q)}_\gamma(|f_\sigma|) \dd\gamma $$
  with $\alpha \in F^\times/F^{\times 2}$, whose absolute convergence follows from Proposition \ref{prop:orbint-bound} .

  From the dominated convergence theorem, one sees $\lim_{k \to \infty} a_k^T = a^T$ where $a^T$ is the expression
  \begin{multline*}
    |W(U(V,q)^\circ, T(F))|^{-1} \int_{T(F)}  |D^{\U(V,q)^\circ}(\gamma)|^{\frac{1}{2}} |D^{\tGL(H)}(\delta)|^{\frac{1}{2}} \\ \iint_{\substack{x \in T(F) \backslash \GL(H) \\ y \in T(F) \backslash \U(V,q)}}
  \sum_{\alpha \in F^\times/F^{\times 2}} \omega_\pi(\alpha)^{-1} f^{\tilde{\pi}}(\alpha x^{-1}\delta x) f_\sigma(y^{-1}\gamma y) \dd x \dd y \dd \gamma .
  \end{multline*}
  Recall that $\text{vol}(Z_{\U(V,q)^\circ}(F) \backslash T(F))=1$, hence the integral over $y$ yields
  $$ \int_{Z_{\U(V,q)^\circ}(F) \backslash U(V,q)} f_\sigma(y^{-1} \delta y) \dd y.$$
  Note that $T(F) \cap F^\times = Z_{\U(V,q)^\circ}(F) \subset \{\pm 1\}$ (as subgroups of $\GL(H)$) in each case. Put $t := [\{\pm 1\} : Z_{\U(V,q)^\circ}(F)]$ which equals $2$ in the odd orthogonal case, otherwise it equals $1$. The integral over $x$ yields
  \begin{multline*}
    \int_{T(F) \backslash \GL(H)} \sum_{\alpha \in F^\times/F^{\times 2}} \omega_\pi(\alpha)^{-1} f^{\tilde{\pi}}(\alpha x^{-1} \delta x) \dd x \\
    = t \cdot \int_{T(F) F^\times \backslash \GL(H)} \; \int_{\{\pm 1\} \backslash F^\times} \sum_{\alpha \in F^\times/F^{\times 2}} \omega_\pi(\alpha)^{-1} f^{\tilde{\pi}}(\alpha z^{-2} x^{-1} \delta x)  \dd z\dd x \\
    = t|2|_F^{-1} \int_{T(F) F^\times \backslash \GL(H)} \int_{F^\times} \omega_\pi(z)^{-1} f^{\tilde{\pi}}(z x^{-1}\delta x) \dd z \dd x \\
    = t |2|_F^{-1} \int_{T(F) F^\times \backslash \GL(H)} f_{\tilde{\pi}}(x^{-1}\delta x)\dd x = t |2|_F^{-1} \int_{F^\times \backslash \GL(H)} f_{\tilde{\pi}}(x^{-1}\delta x)\dd x
  \end{multline*}
  using the fact that $\Ad_\delta(z)=z^{-1}$ and $\omega_\pi^2=1$.

  Summing over $T$ gives $R_\mathrm{ell}(f^{\tilde{\pi}}, f_\sigma) = \Res_{s=0} \sum_{k=0}^\infty q_F^{-2nks} a_k$ where
  $$ \lim_{k \to \infty} a_k = \sum_T a^T = t|2|_F^{-1} R^\text{ell}(f_{\tilde{\pi}},f_\sigma). $$
  To relate the residue at $s=0$ and $\lim_{k \to \infty} a_k$, it remains to apply \cite[Proposition 6]{SS10}.
\end{proof}

Now return to the study of $R^\text{ell}$.
\begin{lemma}\label{prop:Rell-character}
  We have
  \begin{align*}
    R^\mathrm{ell}(f_{\tilde{\pi}}, f_\sigma) &= c \cdot \frac{f_{\tilde{\pi}}^\circ(1) f_\sigma(1)}{d(\pi)d(\sigma)} \sum_T |W(U(V,q)^\circ, T(F))|^{-1} \int_{T(F)} I^{\U(V,q)^\circ}(\sigma,\gamma) I^{\tGL(H)}(\tilde{\pi}, \delta) \dd\gamma \\
    & = c \cdot \frac{f_{\tilde{\pi}}^\circ(1) f_\sigma(1)}{d(\pi)d(\sigma)} \int_{\Gamma_\mathrm{reg,ell}(\U(V,q)^\circ)} I^{\U(V,q)^\circ}(\sigma,\gamma) I^{\tGL(H)}(\tilde{\pi}, \delta) \dd\gamma,
  \end{align*}
  where $c=1$ in the symplectic case, and $c=2$ in the even or odd orthogonal cases.
\end{lemma}
\begin{proof}
  To begin with, we replace the integral over $Z_{\U(V,q)^\circ}(F) \backslash \U(V,q)$ in $R^\text{ell}$ by an integral over $Z_{\U(V,q)^\circ}(F) \backslash \U(V,q)^\circ$ at the cost of introducing $c$. This is feasible as conjugating $\gamma$ by an element of $\U(V,q)$ amounts to conjugating $\delta$ by $\GL(H,F)$, by \eqref{eqn:section}. %Observe that the term $[\GL(H)^\delta : \GL(H)_\delta]^{-1}$ in our definition of orbital integrals disappears by virtue of Proposition \ref{prop:vreg-centralizer}. (no orbital integrals here...)

  In view of the ellipticity of $\gamma$ and $\delta$ (Corollary \ref{prop:GS-ellipticity}), the first equality follows by applying Proposition \ref{prop:orbint-character} to $\U(V,q)^\circ$ and $\tGL(H)$, while the second follows from our definition of the measure on $\Gamma_\mathrm{reg,ell}(\U(V,q)^\circ)$ in \eqref{eqn:measure-1}.
\end{proof}

In the next subsection, we will investigate the even orthogonal case $\dim_F V = 2n$, under the hypothesis that $\pi$ does not come from $\SO(2n+1)$ by endoscopic transfer.

\subsection{The pairing for even orthogonal groups}
The Goldberg-Shahidi-Spallone formalism is now specialized to the even orthogonal case. Assume that $\dim_F V_1 = 6n$, $\dim_F V = \dim_F H = 2n$. In this case $\pi \simeq \pi^\vee$. Set $G := \SO(V,q)$, thus $\sigma \in \Pi_2(G)$. %The following arguments will be based on the crude local Langlands correspondence in \S\ref{sec:crude-LLC}.

Since $\pi \in \Pi_\text{sc,temp}(\GL(H)) \subset \Pi_{2,\text{temp}}(\GL(H))$, we have $\phi \leftrightarrow \pi$ for some selfdual $L$-parameter $\phi \in \Phi_\text{ell,bdd}(\tGL(H))$ which is irreducible as a representation of $\WD_F$. By \eqref{eqn:Phi_2-disjoint}, either
\begin{enumerate}
  \item $\sgn(\phi)=-1$, hence $\pi$ comes from $\SO(2n+1)$, or
  \item $\sgn(\phi)=1$, hence $\pi$ comes from a unique elliptic endoscopic group $G'=\SO(V',q')$, where $(V',q')$ is an $F$-quadratic space with $\dim_F V' = 2n$.
\end{enumerate}

\begin{hypothesis}\label{hyp:even-SO}
  Assume $\pi$ does not come from $\SO(2n+1)$.
\end{hypothesis}

In Theorem \ref{prop:LLC-SO}, we have constructed a subset $\Pi_\phi \subset \overline{\Pi_2}(G') $ and a distribution $\Theta^{G'}_\phi$ on $G'$. Let $s$ be the nontrivial element in $\Out_{2n}(G')$. In view of \eqref{eqn:average}, it satisfies
$$ \Theta^{G'}_\phi(f^{G'}) = \sum_{\bar{\nu} \in \Pi_\phi} \Theta^{G'}_{\nu}\left(\frac{f^{G'}+s(f^{G'})}{2} \right) = \frac{1}{2} \sum_{\bar{\nu} \in \Pi_\phi} (\Theta^{G'}_\nu + \Theta^{G'}_{s\nu})(f^{G'}), \quad f^{G'} \in \mathcal{I}(G'),  $$
where $\nu$ is any inverse image of $\bar{\nu}$ in $\Pi_2(G')$. In particular, $\Theta^{G'}_\phi$ is a virtual character of $G'$ and one can define the multiplicity
\begin{gather}\label{eqn:multiplicity}
\text{mult}(\sigma:\phi) := \begin{cases}
  2 \cdot (\text{the coefficient of } \Theta^G_\sigma \text{ in } \Theta^G_\phi), & \text{ if } G=G', \\
  0, & \text{ if } G \neq G',
\end{cases}
\end{gather}
which belongs to $\{0,1,2\}$.

\begin{theorem}\label{prop:main}
  Under the Hypothesis \ref{hyp:even-SO}, we have
  $$ R^\mathrm{ell}(f_{\tilde{\pi}}, f_\sigma) = \frac{f_{\tilde{\pi}}^\circ(1) f_\sigma(1)}{d(\pi)d(\sigma)} \cdot \gamma_{\psi_F}(2(-1)^n q) \cdot \mathrm{mult}(\sigma:\phi) $$
  for all $f_{\tilde{\pi}} \in \mathcal{A}(\tilde{\pi})$, $f_\sigma \in \mathcal{A}(\sigma)$. In particular, $R^\mathrm{ell}(f_{\tilde{\pi}}, f_\sigma)$ is not identically zero if and only if $\pi$ comes from $G$ and the $\Out_{2n}(G)$-orbit of $\sigma$ is contained $\Pi_\phi$.
\end{theorem}
Bonus: it follows immediately that $R^\text{ell}$ does not depend on the choice of $X \in \Isom_F(V,H^\vee)$.
\begin{proof}
  In view of Lemma \ref{prop:Rell-character}, it suffices to show that
  $$ \int_{\Gamma_\mathrm{reg,ell}(G)} I^G(\sigma,\gamma) I^{\tGL(H)}(\tilde{\pi}, \delta) \dd\gamma = \gamma_{\psi_F}(2(-1)^n q) \cdot 2^{-1}\mathrm{mult}(\sigma:\phi). $$

  Upon restriction to an open dense subset of $\Gamma_\text{reg,ell}(G)$, one may assume that $\gamma$ and $\delta = \delta_X(\gamma)$ are both very regular semisimple. By Theorem \ref{prop:char-identity}, we have
  $$ I^{\tGL(H)}(\tilde{\pi},\delta) = \sum_{\gamma' \in \overline{\Delta}_{\mathrm{reg,ell}}(G')} S^{G'}(\phi,\gamma') \Delta_\lambda(\gamma',\delta). $$

  Since $\delta \GSnorm \gamma$, Corollary \ref{prop:GS-KS} affirms that $\delta$ corresponds to $\gamma^{-1} \in G_\text{reg}(F)$ by twisted endoscopy. If $G \neq G'$, no $\gamma' \in \overline{\Delta}_{\mathrm{reg,ell}}(G')$ can correspond to $\delta$ according to Lemma \ref{prop:separation}, hence $R^\mathrm{ell}(f_{\tilde{\pi}}, f_\sigma)=0=\text{mult}(\sigma:\phi)$.

  Thus we may suppose $G=G'$. Again, Lemma \ref{prop:separation} affirms that $\gamma^{-1}$ represents the unique class in $\overline{\Delta}_{\mathrm{reg,ell}}(G)$ corresponding to $\delta$, hence
  \begin{align*}
    \int_{\Gamma_\mathrm{reg,ell}(G)} I^G(\sigma,\gamma) I^{\tGL(H)}(\tilde{\pi}, \delta) \dd\gamma & = \int_{\Gamma_\mathrm{reg,ell}(G)} I^G(\sigma,\gamma) S^G(\phi, \gamma^{-1}) \Delta_\lambda(\gamma^{-1},\delta) \dd\gamma \\
    & = \gamma_{\psi_F}(2(-1)^n q) \int_{\Gamma_\mathrm{reg,ell}(G)} I^G(\sigma,\gamma) S^G(\phi, \gamma^{-1}) \dd\gamma \\
    & = \gamma_{\psi_F}(2(-1)^n q) \int_{\Gamma_\mathrm{reg,ell}(G)} I^G(\sigma,\gamma) \overline{S^G(\phi, \gamma)} \dd\gamma ,
  \end{align*}
  by Corollary \ref{prop:Delta-formula-GS}. To get the term $2^{-1}\text{mult}(\sigma:\phi)$, it remains to apply Schur's orthogonality relations for $\Pi_2(G)$: see \cite[Theorem 3]{Cl91}. % Or \cite[Theorem 6.1]{Ar93}
\end{proof}
This justifies \cite[Definition 5.1]{GS98}, as promised.
\begin{remark}\label{rem:nonqs}
  In \cite{Sp11}, Spallone also considered the broader setting in which $\SO(V,q)$ is not necessarily quasisplit. It seems that our arguments can be adapted accordingly, once Arthur's endoscopic classification in the non-quasisplit case is completed: see \cite[Chapter 9]{ArEndo} for an announcement of his results.
\end{remark}

%%%%%%%%%%%%%%%
\bibliographystyle{abbrv}
\bibliography{Integration}

\begin{flushleft}
  Wen-Wei Li \\
  Morningside Center of Mathematics, \\
  Academy of Mathematics and Systems Science, Chinese Academy of Sciences, \\
  55, Zhongguancun East Road, \\
  100190 Beijing, People's Republic of China. \\
  E-mail address: \texttt{wwli@math.ac.cn}
\end{flushleft}

\end{document}